\documentclass[12pt]{article}
\usepackage{amssymb}
\usepackage{amsmath}
\usepackage{color}
\usepackage{graphicx,epstopdf}
\usepackage{float}

\RequirePackage[OT1]{fontenc}
\RequirePackage{amsthm,amsmath}
\RequirePackage[numbers]{natbib}
\numberwithin{equation}{section}
\theoremstyle{plain}
\newtheorem{thm}{Theorem}[section]
\newtheorem{example}{Example}[section]
\newtheorem{lemma}{Lemma}[section]
\newtheorem{corollary}{Corollary}[section]
\newtheorem{definition}{Definition}[section]
\newtheorem{remark}{Remark}[section]
\def\ud{\, \mathrm{d}}

\begin{document}

\centerline {\Large{\bf A Representation Theorem for Smooth Brownian
Martingales}}

\centerline{}

\centerline{}

\centerline{\bf {Sixian Jin\footnote{Institute of Mathematical Sciences,
Claremont Graduate University, Sixian.Jin@cgu.edu}, Qidi
Peng\footnote{Institute of Mathematical Sciences, Claremont Graduate
University, Qidi.Peng@cgu.edu}, Henry Schellhorn\footnote{Corresponding
author, Institute of Mathematical Sciences, Claremont Graduate University,
Henry.Schellhorn@cgu.edu} }}

\centerline{}

\noindent\textbf{Abstract:} We show that, under certain smoothness conditions, a
Brownian martingale, when evaluated at a fixed time, can be represented via
an exponential formula at a later time. The time-dependent generator of this
exponential operator only depends on the second order Malliavin derivative
operator evaluated along a "frozen path". The exponential operator can be
expanded explicitly to a series representation, which resembles the Dyson
series of quantum mechanics. Our continuous-time martingale representation
result can be proven independently by two different methods. In the first
method, one constructs a time-evolution equation, by passage to the limit of
a special case of a \textit{backward Taylor expansion} of an approximating
discrete time martingale. The exponential formula is a solution of the
time-evolution equation, but we emphasize in our article that the
time-evolution equation is a separate result of independent interest. In the
second method, we use the property of denseness of exponential functions. We provide several applications of the exponential
formula, and briefly highlight numerical applications of the backward Taylor
expansion.



\section{Introduction}

The problem of representing Brownian martingales has a long and
distinguished history. Dambis \cite{r2} and Dubins-Schwarz \cite{r3} showed
that continuous martingales can be represented in terms of time-changed
Brownian motions. Doob \cite{r4}, Wiener and It\^o developed what is often
called It\^o's martingale representation theorem: every local Brownian
martingale has a version which can be written as an It\^o integral plus a
constant. In this article, we consider a special kind of martingales which
are conditional expectations of an $\mathcal{F}_{T}$-measurable random
variable $F$. Recall that, when the random variable $F$ is Malliavin
differentiable, the Clark-Ocone formula (\cite{r1,r9}) states that the
integrand in It\^o's martingale representation theorem is equal to the
conditional expectation of the Malliavin derivative of $F$. We focus on a
less general case, where the Brownian martingale is assumed to be
"infinitely smooth". Namely, the target random variable $F$ is infinitely
differentiable in the sense of Malliavin. We show that such a Brownian
martingale, when evaluated at time $t\le T$, $E[F|\mathcal{F}_{t}]$, can be
represented as an exponential operator of its value at the later time $T$.

While smoothness is a limitation to our result, our representation
formula opens the way to new numerical schemes, and some analytical
asymptotic calculations, because the exponential operator can be calculated
explicitly in a series expansion, which resembles the Dyson series of
quantum mechanics. Although we still call our martingale's expansion Dyson
series, there are two main differences between our martingale representation
and the Dyson formula for the initial value problem in quantum mechanics.
First, in the case of martingales, time flows backward. Secondly, the
time-evolution operator is equal to one half of the second-order Malliavin
derivative evaluated along a constant path, while for the initial value
problem in quantum mechanics the time-evolution operator is equal to $-2\pi
i $ times the time-dependent Hamiltonian divided by the Planck constant.

Our continuous-time martingale representation result can be proved using two
different methods: by discrete time approximation and by
approximation from a dense subset. In the first method, the key idea is to
construct the \textit{backward Taylor expansion} (BTE) of an approximating
discrete-time martingale. The BTE was introduced in Schellhorn and Morris
\cite{r13}, and applied to price American options numerically. The idea in
that paper was to use the BTE to approximate, over one time-step, the
conditional expectation of the option value at the next time-step. While not
ideal to price American options because of the lack of differentiability of
the payoff, the BTE is better suited to the numerical calculation of the
solution of smooth backward stochastic differential equations (BSDE). In a
related paper, Hu et al. \cite{r5} introduce a numerical scheme to solve a
BSDE with drift using Malliavin calculus. Their scheme can be viewed as a
Taylor expansion carried out until the first order. Our BTE can be seen as a
generalization to higher order of that idea, where the Malliavin derivatives
are calculated at the future time-step rather than at the current time-step.

The time-evolution equation results then by a passage to the limit, when the
time increment goes to zero, of the BTE, following the "frozen path". The
exponential formula is then a solution of the time-evolution equation, under
certain conditions. We stress the fact that both the BTE and the
time-evolution equation are interesting results in their own right. Since
the time-evolution equation is obtained from the BTE only along a particular
path, we conjecture that there might be other types of equations that smooth
Brownian martingales satisfy in continuous time. The time-evolution equation
can also be seen as a more general result than the exponential formula, in
the same way that the semi-group theory of partial differential equations
does not subsume the theory of partial differential equations. For instance,
other types of series expansion, like the Magnus expansion \cite{OteoRos}
can be used to solve a time-evolution equation.

We also sketch an alternate method, which we call the \textit{density method}
of proof of the exponential formula, which uses the denseness of stochastic
exponentials in $L^{2}(\Omega )$. The complete proof \footnote{%
This proof is available from the authors, upon request.} goes along the
lines of the proof of the exponential formula for fractional Brownian motion
(fBm) with Hurst parameter $H>1/2$, which we present in a separate paper
\cite{JPS}. We emphasize that it is most likely nontrivial to obtain the
exponential formula in the Brownian case by a simple passage to the limit of
the exponential formula for fBm when $H$ tends to $1/2$ from above. We
mention three main differences between Brownian motion and fBm in our
context. First, by the Markovian nature of Brownian motion, the backward
Taylor expansion leads easily in the Brownian case to a numerical scheme.
Second, there is a time-evolution equation in the Brownian case, but
probably not in the fBm case, so that the BTE method of proof is
unavailable. Third, the fractional conditional
expectation (which is defined only for $H>1/2$ in \cite{FB}) in general does not coincide with
the conditional expectation.

The structure of this paper is the following. We first expose the discrete
time result, namely the backward Taylor expansion for functionals of
discrete Brownian sample path. We then move to continuous time, and present
the time-evolution equation and exponential formula. We then sketch the
density method of proof. Four explicit examples are given, which show the
usefulness of the Dyson series in analytic calculations. Example 4 is about
the Cox-Ingersoll-Ross model with time-varying parameters, which, as far as
we know, is a new result. All proofs of main results are relegated to the
appendix.

\section{Martingale Representation}

\subsection{\noindent Preliminaries and notation}

This section reviews some basic Malliavin calculus and
introduces some definitions that are used in our article. We denote by $%
(\Omega, \mathcal{\{F}_{t}\}_{t \geq 0},P)$ a complete filtered
probability space, where the filtration $\mathcal{\{F}_{t}\}_{t \geq 0}$ satisfies the
usual conditions, i.e., it is the usual augmentation of the filtration
generated by Brownian motion $W$ on $\mathbb{R}_+$ (most results can be easily
generalized to Brownian motion on $\mathbb{R}_+^{d}$, $d\ge2$). Unless stated otherwise all equations involving
random variables are to be regarded to
hold $P$-almost surely.

Following by \cite{r10}, we say that a real function $g:[0,T]\rightarrow
\mathbb{R}^{n}$ is symmetric if:
\begin{equation*}
g(x_{\sigma (1)},\ldots,x_{\sigma(n)})=g(x_{1},\ldots,x_{n}),
\end{equation*}%
for any permutation $\sigma $ on $(1,2,\ldots,n)$. If in addition, $g\in
L^2([0,T]^n)$, i.e.,
\begin{equation*}
||g||_{L^{2}([0,T]^{n})}^{2}=\int_{0}^{T}\ldots
\int_{0}^{T}g^{2}(x_{1},\ldots,x_{n})\, \mathrm{d} x_{1}\ldots\, \mathrm{d}
x_{n}<\infty,
\end{equation*}%
then we say $g$ belongs to $\hat{L}^{2}([0,T]^{n})$, the space of symmetric
square-integrable functions on $[0,T]^{n}$. Denote by $L^2(\Omega)$ the
space of square-integrable random variables, i.e., the norm of $F\in
L^2(\Omega)$ is
\begin{equation*}
\|F\|_{L^2(\Omega)}=\sqrt{E[F^2]}<\infty.
\end{equation*}
The Wiener chaos expansion of $F\in L^2(\Omega)$, is thus given by
\begin{equation*}
F=\sum_{m=0}^{\infty }I_{m}(f_{m})\text{ \ \ in }L^{2}(\Omega ),
\label{WienerChaos}
\end{equation*}%
where $\{f_{m}\}_{m\ge0}$ is a uniquely determined sequence of
deterministic functions ($f_m:~\mathbb R^m\rightarrow\mathbb R$ is the so-called $m$-dimensional kernel) with $f_0\in\mathbb R$, $ f_m\in\hat{L}^{2}([0,T]^{m})$ for $m\ge1$, and the operator $I_m:~\hat{L}^{2}([0,T]^{m})\rightarrow L^2(\Omega)$ is
defined as
\begin{equation*}
\left\{%
\begin{array}{lll}
I_{0}(f_{0}) & = & f_{0};\\
I_{m}(f_{m}) & = & m!\int_{0}^{T}\int_{0}^{t_{m}}\ldots
\int_{0}^{t_{2}}f(t_{1},\ldots,t_{m})\, \mathrm{d} W(t_{1})\, \mathrm{d}
W(t_{2})\ldots\, \mathrm{d} W(t_{m}),~\text{ for }m\ge1.
\end{array}%
\right.
\end{equation*}
For an $L^2(\Omega)$ element $u$, we denote its Skorohod integral by $\int\limits_{0}^{T}u(s)\delta W(s)$, which is considered as the adjoint of the
Malliavin derivative operator. To be more explicit, it can be defined this way: for all $t\in%
[0,T]$, if the Wiener chaos expansion of $u(t)$ is
\begin{equation*}
u(t)=\sum_{m=0}^{\infty }I_{m}(f_{m}(\cdot,t))~\mbox{in $L^2(\Omega)$},
\end{equation*}
where for each $m\ge0$, $f_m:~\mathbb R^{m+1}\rightarrow\mathbb R$ is a uniquely determined $(m+1)$-dimensional kernel, then the Skorohod integral of $u$ is defined to be
\begin{equation*}
\int\limits_{0}^{T}u(s)\delta W(s)=\sum_{m=0}^{\infty }I_{m+1}(\tilde{f}%
_{m})~\mbox{in $L^2(\Omega)$},
\end{equation*}
where $\tilde{f}_{m}$ denotes the symmetrization of the $(m+1)$-dimensional
kernel $f_{m}$ with respect to its $(m+1)$th argument (see Proposition 1.3.1
in \cite{r8} for more details):
\begin{equation*}
\tilde{f}_{m}(t_1,\ldots,t_m,t)=\frac{1}{m+1}\Big(f_m(t_1,\ldots,t_m,t)+%
\sum_{i=1}^mf_m(t_1,\ldots,t_{i-1},t,t_{i+1},\ldots,t_m,t_i)\Big).
\end{equation*}
Following Lemma 4.16 in \cite{r10}, the Malliavin derivative $D_{t}F$ of $F$
(when it exists) satisfies
\begin{equation*}
D_{t}F=\sum_{m=1}^{\infty }mI_{m-1}(f_{m}(\cdot,t))~\mbox{in}~
L^{2}(\Omega).
\end{equation*}
We denote the Malliavin derivative of order $l$ of $F$ at time $t$ by $%
D_{t}^{l}F$, as a shorthand notation for $\underbrace{D_t\ldots D_t}_{l~%
\mbox{times}}F$. We call $\mathbb{D}_{\infty }([0,T])$ the set of random
variables which are infinitely Malliavin differentiable and $\mathcal{F}_{T}$%
-measurable. A random variable is said to be infinitely Malliavin
differentiable if $F\in L^2(\Omega)$ and for any integer $n\ge1$,
\begin{equation}  \label{dinfty}
\Big\|\sup_{s_{1},\ldots ,s_{n}\in [0,T]}\big|%
D_{s_{n}}\ldots D_{s_{1}}F\big|\Big\|_{L^2(\Omega)} <\infty .
\end{equation}
In particular, we denote by $\mathbb{D}^N([0,T])$ the collection of all $F\in L^2(\Omega)$ satisfying (\ref{dinfty}) for $n\leq N$.

We define the freezing path operator $\omega ^{t}$ on a Brownian motion $W$ by
\begin{equation}  \label{defineomegat}
\left\{W(s,\omega ^{t})\right\}_{s\ge0}:=\left\{W(s)\chi_{[s\le t]}+W(t)\chi_{[s>t]}\right\}_{s\ge0},
\end{equation}
with $\chi$ being the indicator function. When Brownian motion is defined as the coordinate mapping process (see \cite%
{KS}), then each trajectory of $W(\cdot,\omega ^{t})$ represents obviously a "frozen path" -- a
particular path where the corresponding Brownian motion becomes constant
after time $t$. More generally, let $F\in L^2(\Omega)$ be a random variable generated by $\{W(s)\}_{s\in[0,T]}$, namely, there exists an operator such that $F=G(W\chi_{[0,T]})$ and $\{u(t):=G(W\chi_{[0,t]})\}_{t\ge0}$ is a continuous-time process in $L^2(\Omega)$, the freezing path operator $\omega^t$ on $F$ is then defined by
$$
F(\omega^t)=G(W\chi_{[0,T\wedge t]}),
$$
where $T\wedge t:=\min\{T,t\}$. In the following we denote by
$$
<f,W\chi_{[0,T]}>=\int_0^Tf(s)\ud W(s).
$$
Then for example, let $F=G(W\chi_{[0,T]})=\big(<1,W\chi_{[0,T]}>\big)^2=W(T)^2$, then
$$
F(\omega^t)=G(W\chi_{[0,T\wedge t]})=\big(<1,W\chi_{[0,T\wedge t]}>\big)^2=W(T\wedge t)^2.
$$
Remark that  if $\lim_{M\rightarrow\infty}F_M= F=G(W\chi_{[0,T]})$ in $L^2(\Omega)$ and $F_M=G_M(W\chi_{[0,T]})$,  then by definition of freezing path operator,
$$G_M(W\chi_{[0,t]})\xrightarrow[M\rightarrow\infty]{L^2(\Omega)} G(W\chi_{[0,t]})~\mbox{for all}~t\in\mathcal I\subset[0,T]$$
 is equivalent to
\begin{equation}
\label{conv:omega}
F_M(\omega^t)\xrightarrow[M\rightarrow\infty]{L^2(\Omega)}F(\omega^t),~\mbox{for}~t\in\mathcal I.
\end{equation}
The freezing path operator is obviously linear, however it is generally not preserved by Malliavin differentiation. For instance, let $F=\frac{1}{2}
W(T)^{2}$. Then for $t\le T$,
$$(D_sF)(\omega ^{t})=\chi_{[s\leq T]}W(t)\neq D_s\left(F(\omega^t)\right)=\chi_{[s\leq t]}W(t).$$
 It is very important to note that $\omega^t$ should be regarded as a left-operator, namely,
\begin{equation*}
F(\omega ^{t})= \omega ^{t}\circ F.
\end{equation*}
We give in the remarks below some examples of illustrative computations on the frozen path, which can also be
viewed as a constructive definition of the operator.

\begin{remark}
\label{freezing}
We show hereafter the freezing path transformation of some random variables with particular forms. Let $t\le T$:
\begin{enumerate}
\item For a polynomial $p$, suppose $%
F=p(W(s_{1}),\ldots,W(s_{n}))$ with $0\le s_1\le\ldots\le s_n\le T$, then $F(\omega ^{t})=p(W(s_{1}\wedge
t),\ldots,W(s_{n}\wedge t))$.
\item The following equations hold:%
\begin{eqnarray*}
&&\Big( \int_{0}^{T}f(s)\, \mathrm{d} W(s)\Big) (\omega
^{t})=\int_{0}^{t}f(s)\, \mathrm{d} W(s)~\mbox{for}~f\in L^2([0,T]); \\
&&\Big( \int_{0}^{T}W(s)\, \mathrm{d} s\Big) (\omega
^{t})=\int_{0}^{t}W(s)\, \mathrm{d} s+W(t)(T-t).
\end{eqnarray*}%
\item For a general It\^o integral, there is not yet a satisfactory or
general approach to compute its frozen path so far. The closed form can be derived if the integral is transformed to an elementary function of Brownian motions. For example, by It\^o formula we can get
\begin{equation*}
\Big( \int_{0}^{T}W(s)\, \mathrm{d} W(s)\Big)(\omega ^{t})=\Big( \frac{%
W(T)^{2}-T}{2}\Big)(\omega ^{t})=\frac{W(t)^{2}-T}{2}.
\end{equation*}
\item Let $F_1,\ldots,F_n\in L^2(\Omega)$ be $\mathcal F_T$-measurable and let $g:~\mathbb R^n\rightarrow\mathbb R^m$  be a continuous function. Then we have for $t\le T$,
    $$
    g(F_1,\ldots,F_n)(\omega^t)=g(F_1(\omega^t),\ldots,F_n(\omega^t)).
    $$
\end{enumerate}
\end{remark}

\subsection{Backward Taylor Expansion (BTE)}
Through this subsection, we assume that $F\in
\mathbb{D}_{\infty }([0,M\Delta])$ (with integer $M\ge1$ and real number $\Delta>0$) is some cylindrical function of Brownian motions. In other words, $F$ has the form $g(W(\Delta ),W(2\Delta),\ldots,W(M\Delta))$ with $g:%
\mathbb{R}^{M}\rightarrow \mathbb{R}$ being some deterministic infinitely differentiable function.

We now present the BTE of the Brownian martingale
evaluated at time $m\Delta$, $0\le m\le M$. First recall that $h_n$, the Hermite polynomial of
degree $n\ge 0$, is defined by $h_0\equiv1$ and
\begin{equation}  \label{hermite}
h_n (x)=(-1)^n\exp \left( \frac{x^2}{2}\right)\frac{\, \mathrm{d}^n}{\,
\mathrm{d} x^n}\exp\left(-\frac{x^2}{2}\right)~~\mbox{for $n\ge1$, $x\in
\mathbb{R}$.}
\end{equation}
For a real value $x$, we define its floor number by
\begin{equation}
\lfloor x\rfloor :=\max \{m\in \mathbb{Z};m\leq x\}.  \label{floor}
\end{equation}
\begin{thm}
\label{BTE} If $F$ satisfies, for each $m\in\{0,\ldots,M-1\}$,
\begin{equation}
\label{Condi}
\sum_{i=0}^{L}\left\| D_{(m+1)\Delta }^{2L-i}F\right\|_{L^2(\Omega)} ^{2} {%
\binom{L}{i}}^{4}\frac{i!}{\left( L!\right) ^{2}}\Delta ^{2L-i}%
\xrightarrow[L\rightarrow\infty]{}0,
\end{equation}%
then for each $m\in\{0,\ldots,M-1\}$,
\begin{equation}
\label{aga2}
E[F|\mathcal{F}_{m\Delta }]=\sum_{l=0}^{\infty }\gamma (m,l)E[D_{(m+1)\Delta
}^{l}F|\mathcal{F}_{(m+1)\Delta }]~\mbox{in}~L^2(\Omega),
\end{equation}%
where $\gamma (m,0)=1$ and for $l\ge1$,
\begin{equation}
\gamma (m,l)=(-1)^{l}\Delta^{l/2}\sum_{j=0}^{\lfloor l/2\rfloor }\frac{1}{%
j!(l-2j)!}h_{l-2j}\left( \frac{W((m+1)\Delta )-W(m\Delta )}{\sqrt{\Delta}}%
\right).  \label{Form_gamma}
\end{equation}
\end{thm}
Here are some remarks on Condition (\ref{Condi}).
\begin{remark}
\label{condition2.6}
\begin{enumerate}
\item
A quite large range of random variables fits
Condition (\ref{Condi}). For instance, by applying Stirling's approximation
to the factorials, we can show that:
\begin{equation*}
\sum_{i=0}^{L}{\binom{L}{i}}^{4}\frac{i!}{\left( L!\right) ^{2}}\Delta
^{2L-i}\leq \frac{\alpha ^{L}}{L^{L}}
\end{equation*}
for some constant $\alpha>0$, which does not depend on $L$. Thus
Condition (\ref{Condi}) is satisfied by $F$, if there is some constant $c>0$
such that for any integers $L$ and any $m=0,1,\ldots,M-1$,
\begin{equation}
\left\|D_{(m+1)\Delta }^{L}F\right\|_{L^2(\Omega)}\leq c^{L}.
\label{sufficondi}
\end{equation}
A simple example that satisfies (\ref{sufficondi}) is $F=e^{\xi W_{M\Delta}}$, for any $\xi\in\mathbb R$.

\item Condition (\ref{Condi}) is clearly satisfied if all the Malliavin derivatives of $F$ of order $l\geq L$ vanish for some $L$. At the meanwhile, the BTE (\ref{aga2}) becomes a finite sum of $L$ terms.
    \end{enumerate}
\end{remark}
Below is an illustrative example to show how to apply Theorem \ref{BTE} to derive the explicit form of a Brownian martingale.
\begin{example}
\label{EX1} Let $F=W((m+1)\Delta)^3$. It is well-known that $E[F|\mathcal{F}%
_{m\Delta}]=W(m\Delta)^3+3\Delta W(m\Delta)$. Now we derive this result by the BTE approach.
 \end{example}
 Since $D_{(m+1)\Delta}^lF=0$ for $l\ge4$,  then by Remark \ref{condition2.6}, Condition (\ref{Condi}) is verified. In view of (\ref{aga2}), to obtain the BTE of $E[F|\mathcal{F}%
_{m\Delta}]$ it remains to compute $\{\gamma(m,l)\}_{l=1,2,3}$. From (\ref{Form_gamma}) we see that:
\begin{eqnarray*}
&&\gamma (m,1)=-\left( W((m+1)\Delta)-W(m\Delta)\right); \\
&&\gamma (m,2)=\frac{1}{2}(W((m+1)\Delta)-W(m\Delta))^2+\frac{\Delta}{2}; \\
&&\gamma (m,3)=-\frac{1}{6}(W((m+1)\Delta)-W(m\Delta))^3-\frac{\Delta}{2}(W((m+1)\Delta)-W(m\Delta)).
\end{eqnarray*}
By Theorem \ref{BTE} and some algebraic computations, we get:
\begin{eqnarray*}
&&E[W((m+1)\Delta)^3|\mathcal{F}_{m\Delta}] \\
&&=\gamma(m,0)W((m+1)\Delta)^3+3\gamma(m,1)W((m+1)\Delta)^2+6\gamma(m,2)W((m+1)\Delta)+6\gamma(m,3) \\
&&=W(m\Delta)^3+3\Delta W(m\Delta).
\end{eqnarray*}
Note that (\ref{aga2}) is a one-step backward time equation. A multiple step BTE expression of $E[F|\mathcal{F}_{m\Delta }]$ can be derived by
applying (\ref{aga2}) recursively:

\begin{corollary}
Let $F$ satisfy Condition (\ref{Condi}) for each $m=0,\ldots,M-1$, then we have for each $m$,
\begin{equation}
E[F|\mathcal{F}_{m\Delta }]=\sum_{j_{m+1}=0}^{\infty
}\ldots\sum_{j_{M=0}}^{\infty }\Big(\prod\limits_{k=m+1}^{M}\gamma
(k-1,j_{k})\Big)D_{(m+1)\Delta }^{j_{m+1}}\ldots D_{M\Delta }^{j_{M}}F.
\label{glasmucht}
\end{equation}
\end{corollary}
We proceed now to discuss
two applications of the BTE. Since space is limited, we describe these
informally.

\begin{description}
\item \textbf{Application to solving FBSDEs}
\end{description}

Consider a forward-backward stochastic differential equation(FBSDE), where the
problem is to find a triplet of adapted processes $\{(X(t),Y(t),Z(t))\}_{t\in[0,T]}$ such
that:
\begin{equation}
\label{FBSDE}
\left\{%
\begin{array}{rll}
\, \mathrm{d} X(t) & = & b(t,X(t))\, \mathrm{d} t+\sigma (t,X(t))\, \mathrm{d%
} W(t) \\
\, \mathrm{d} Y(t) & = & h(t,X(t),Y(t))\, \mathrm{d} t+Z(t)\, \mathrm{d} W(t)
\\
X(0) & = & x \\
Y(T) & = & Q(X),%
\end{array}%
\right.
\end{equation}
where $b$, $\sigma$, $h$ are given deterministic continuous functions; $Q$ is a given function on the path of $X$ from time $0$ to $T$. This
problem is at the same time more general (because of the path-dependency of $%
Q$) and less general than a standard FBSDE (because the coefficients of the
diffusion $X$ do not depend on $Y$ or $Z$). We define
\begin{equation}
U(t):=Y(t)-\int_{0}^{t}h(s,X(s),Y(s))\, \mathrm{d} s.  \label{integral}
\end{equation}
We take the following steps to numerically solve this FBSDE.
\begin{description}
\item[Step 1] From the first equation in (\ref{FBSDE}), we generate a discrete path of $X$, by using the Euler scheme (see e.g. \cite{Desmond}). Denote this path by $\{X(m\Delta)\}_{m=1,\ldots,M\Delta}$, with $M\Delta=T$.
\item[Step 2] Along the same path $\{X(m\Delta)\}_{m=1,\ldots,M\Delta}$, we use (\ref{glasmucht}) to compute
\begin{equation*}
U(m\Delta)=E[Q(X)|\mathcal{F}_{m\Delta }].
\end{equation*}
Thus a scenario of $\{U(m\Delta)\}_{m=1,\ldots,M}$ is obtained.
\item[Step 3] Again along the same path $\{X(m\Delta)\}_{m=1,\ldots,M\Delta}$, we evaluate
$$
Y(M\Delta)=Q_M(X(\Delta),\ldots,X(M\Delta)),
$$
where $\{Q_M\}_{M\geq 1}$ is some sequence of polynomials such that
$$
Q_M(X(\Delta),\ldots,X(M\Delta))\xrightarrow[M\rightarrow\infty]{L^2(\Omega)}Q(X).
$$
\item[Step 4] Use the discrete time equation
$$
U(m\Delta)=Y(m\Delta)-\sum_{i=1}^mh(X(i\Delta),Y(i\Delta),i\Delta)\Delta,~\mbox{for}~m=1,\ldots,M
$$
to establish the backward difference equation
$$
\!\!\!\!\!\!\!\!\!\!\left\{%
\begin{array}{ll}
&Y((m-1)\Delta)= Y(m\Delta)-h(X(m\Delta),Y(m\Delta),m\Delta)\Delta+U((m-1)\Delta)-U(m\Delta),\\
&\ \ \ \mbox{for}~m=1,\ldots,M;\\
&Y(M\Delta)=Q_M(X(\Delta),\ldots,X(M\Delta)).
\end{array}%
\right.
$$
\end{description}
The main numerical difficulty is to evaluate the Malliavin derivatives in Step 2. We can apply the change of variables defined in \cite{DGR} to calculate the Malliavin derivatives of $X$, and then (mutatis mutandis) the Faa di\
Bruno formula for the composition of $Q$ with $X$. We must of course
calculate finite sums instead of infinite ones in (\ref{glasmucht}). We
could then imagine a scheme where, at each step, the "optimal path" is
chosen so as to minimize the global truncation error. We leave all these
considerations for future research.

\begin{description}
\item \textbf{Application to Pricing Bermudan Options}
\end{description}

Casual observation of (\ref{glasmucht}) shows that the complexity of
calculations grows exponentially with time. This shortcoming of the BTE does
not occur in the problem of Monte Carlo pricing Bermudan options, where the
BTE can be competitive as we hint now. Most of the computational burden in
Bermudan option pricing consists in evaluating:
\begin{equation}
C(m\Delta )=E[\max (C((m+1)\Delta ),h((m+1)\Delta ))|\mathcal{F}_{m\Delta
}],  \label{AOP}
\end{equation}
where $C(m\Delta )$ and $h(m\Delta )$ are respectively the continuation
value and the exercise value of the option. The data in this problem consists in the
exercise value at all times and the continuation value $C(M\Delta )$ at
expiration. The conditional expectation $C(m\Delta )$ must be evaluated at
all times $m\Delta $, with $0\leq m<M$ and along every scenario. In
regression-based algorithms, such as \cite{LS} and \cite{TVR}, the
continuation value is regressed at each time $m\Delta $ on a basis of
functions of the state variables, so that $C(m\Delta )$ can be expressed as
a formula. The formula is generally a polynomial function of the previous
values of the state variable. This important fact,
that $C(m\Delta )$ is available formulaically rather than numerically,
makes possible the use of symbolic Malliavin differentiation.\footnote{%
A problem with this approach is the calculation of succesive Malliavin derivatives of the maximum in (\ref{AOP}). We will show in another article how one can
use the conditional expectation $E[C((m+1)\Delta) |\mathcal{F}_{m\Delta }]$
as a control variate, where the control variate is calculated using the BTE.} Since the formula is a polynomial, there is no truncation error in (\ref%
{glasmucht}) if the state variables are Brownian motion.

\subsection{The Time-evolution Equation}
A non-intuitive feature of the BTE is that \textit{any path} from $t$ to $T$
can be chosen to approximate conditional expectations evaluated at time $t$,
as opposed to Monte Carlo simulation where many paths are needed. In this subsection we will choose the frozen paths $\omega^t$ for $t\leq m\Delta $ to derive our second main result.
\bigskip For notational simplicity we define:%
\begin{equation*}
M^{F}(t)=E[F|\mathcal{F}_{t}].
\end{equation*}
Define the set of random variables $$\mathcal{M}(t):=\left\{ M^{F}(t):F\in
\mathbb D^6([0,T])~\mbox{and $F$ only depends on a discrete Brownian path}
\right\} .$$  Let the time-evolution operator $P_{s}$ be a conditional
expectation operator, restricted to some particular subset of $L^2(\Omega)$. More precisely, we define for any time $%
s\in [0,T]$,
\begin{eqnarray}
P_{s}~:~\bigcup\limits_{\tau \in [ s,T]}\mathcal{M}(\tau )
&\longrightarrow &\mathcal{M}(s)  \notag \\
M^{F}(\tau ) &\longmapsto &M^{F}(s).  \label{Def_of_P}
\end{eqnarray}%
For example, if $F$ is $\mathcal F_T$-measurable, then $P_sF=E[F|\mathcal F_{s}]$ for $s\le T$. We also define the time-derivative of the time-evolution operator as: for $%
0<s\leq \tau \leq T$, provided the limit exists in $L^{2}(\Omega )$:
\begin{equation}
\frac{\,\mathrm{d}P_{s}}{\,\mathrm{d}s}(M^{F}(\tau )):=\lim_{h\downarrow 0}%
\frac{P_{s}(M^{F}(\tau ))-P_{s-h}(M^{F}(\tau ))}{h}~\mbox{in $L^2(\Omega)$}.
\label{Def_of_DP}
\end{equation}
Below we state the time-evolution equation, which plays a key role in the proof of Theorem \ref{ExpFormula}.


\begin{thm}
\label{TEE} Let $s\in (t,T]$. Suppose $F \in \mathbb{D}^6 ([0,T])$ and only depends on a discrete Brownian path $W(\Delta),\ldots,W(M\Delta )$ ($T=M\Delta $). Then
the operator $P_{s}$ satisfies the following equation, whenever the right
hand-side is an element in $L^2(\Omega)$:%
\begin{equation}
\left(\frac{\, \mathrm{d} P_{s}F}{\, \mathrm{d} s}\right)(\omega ^{t})=-%
\frac{1}{2}\left( D_{s}^{2}P_{s}F\right) (\omega ^{t}).  \label{TEEE}
\end{equation}%
Note that this equality holds for each $s\in (t,T]$ in $L^{2}(\Omega )$.
\end{thm}

We can see the analogy between our time-evolution operator $P_{s}$ and the
one in quantum mechanics. The difference is that in quantum mechanics $%
-(1/2)D^{2}$ is replaced by the Hamiltonian divided by $-i\hslash $. The next theorem will provide a Dyson series solution to Equation (\ref{TEEE}%
).




\subsection{Dyson Series Representation}

\noindent For esthetical reasons we introduce a "chronological operator". In
this we follow Zeidler \cite{r15}. Let $\{H(t)\}_{t\geq 0}$ be a collection of
operators. The chronological operator $\mathcal{T}$ is defined by
\begin{equation*}
\mathcal{T(}H(t_{1})H(t_{2})\ldots H(t_{n})):=H(t_{1^{\prime
}})H(t_{2^{\prime }})\ldots H(t_{n^{\prime }}),
\end{equation*}
\noindent where $(t_{1^{\prime }},\ldots,t_{n^{\prime }})$ is a permutation of
$(t_{1},\ldots,t_{n})$ such that $t_{1^{\prime }}\geq t_{2^{\prime }}\geq
\ldots\geq t_{n^{\prime }}$.

For example, it is showed in Zeidler \cite{r15} on Page 44-45 that:
\begin{equation*}
\int_{0}^{t}\int_{0}^{t_{2}}H(t_{1})H(t_{2})\, \mathrm{d} t_{1}\, \mathrm{d}
t_{2}=\frac{1}{2!}\int_{0}^{t}\int_{0}^{t}\mathcal{T(}H(t_{1})H(t_{2}))\,
\mathrm{d} t_{1}\, \mathrm{d} t_{2}.
\end{equation*}
This will be the only property of the chronological operator we will use in
this article.
\begin{definition}
We define the exponential operator of a time-dependent generator $H$ by
\begin{equation}
\mathcal{T}\exp \Big(\int_{t}^{T}H(s)\, \mathrm{d} s\Big)=\sum_{k=0}^{%
\infty }\int_{t}^{T}\int_{t_1}^T\ldots\int_{t_{k-1}}^{T}H(t
_{1})\ldots H(t_{k})\, \mathrm{d}t _{k}\ldots\, \mathrm{d}t_{1}.
\label{Dyson}
\end{equation}
\end{definition}
In quantum field theory, the series on the right hand-side of (\ref{Dyson})\
is called a \textit{Dyson series} (see e.g. \cite{r15}).

\begin{thm}
\label{ExpFormula} Let $F\in \mathbb{D}_{\infty }([0,T])$
satisfy
\begin{equation}
\frac{(T-t)^{n}}{2^{n}n!}\Big\|\sup_{u_{1},...,u_{n}\in
\lbrack t,T]}\left\vert (D_{u_{n}}^{2}\ldots D_{u_{1}}^{2}F)(\omega
^{t})\right\vert \Big\|_{L^2(\Omega)} \xrightarrow[n\rightarrow \infty]{}0
\label{assumptionb}
\end{equation}%
for some fixed $t\in \lbrack 0,T]$. Then
\begin{equation}
E[F|\mathcal{F}_{t}]=\Big( \mathcal{T}\exp \Big( \frac{1}{2}%
\int_{t}^{T}D_{s}^{2}\,\mathrm{d}s\Big) F\Big) (\omega ^{t})~\mbox{in}%
~L^{2}(\Omega ).  \label{expo_formula}
\end{equation}
\end{thm}
The importance of the exponential formula (\ref{expo_formula}) stems from
the Dyson series representation (\ref{Dyson}). By the property of symmetry of the function $(s_1,\ldots,s_i)\longmapsto D_{s_{i}}^{2}%
\ldots D_{s_{1}}^{2}F$ , we are able to rewrite the Dyson series hereafter in
a more convenient way:
\begin{eqnarray}
\label{DysonSeries}
E\left[ F|\mathcal{F}_{t}\right]&=&\sum_{i=0}^{\infty }\frac{1}{2^{i}}\int_{t\leq s_{1}\leq \ldots \leq
s_{i}\leq T}(D_{s_{i}}^{2}\ldots D_{s_{1}}^{2}F)(\omega ^{t})\,\mathrm{d}%
s_{i}\ldots \,\mathrm{d}s_{1}  \notag \\
&=&\sum_{i=0}^{\infty }\frac{1}{2^{i}i!}\int_{[t,T]^{i}}(D_{s_{i}}^{2}%
\ldots D_{s_{1}}^{2}F)(\omega ^{t})\,\mathrm{d}s_{i}\ldots \,\mathrm{d}s_{1},
\end{eqnarray}
in which the first term is $F(\omega^t)$ by convention.
\begin{example}
\label{EX2} We retake Example \ref{EX1} with a new approach of computation. Namely, we apply the exponential formula approach to compute $E[W(T)^{3}|\mathcal{F}_{t}]$, $t\le T$.
\end{example}
Let $F=W(T)^{3}$. Observe that for $t\le s_1, s_2\le T$,
\begin{equation}
\label{computeF}
F(\omega ^{t}) =W(t)^{3},~(D_{s_1}^{2}F)(\omega ^{t}) =6W(t)~\mbox{and}~(D_{s_2}^{2}D_{s_1}^2F)(\omega ^{t})=0.
\end{equation}
It follows from (\ref{DysonSeries}) and (\ref{computeF}) that
$$
E[F|\mathcal{F}_{t}]=F(\omega ^{t})+\frac{1}{2}\int_{t}^{T}(D_{s_1}^{2}F)(\omega ^{t})\,\mathrm{d}s_1=W(t)^{3}+3(T-t)W(t).
$$
We will use (\ref{DysonSeries}) for some more analytical calculations, as we show in the next
subsection. The analytical calculations become quickly nontrivial, though,
and, for numerical applications, one may want to develop an automatic tool
that performs symbolic Malliavin differentiation (see the earlier
discussion, on the implementation of the BTE).
\begin{remark} As mentioned in the introduction, there is
another way of proving the exponential formula, by the so-called density
method (the denseness of the exponential functions).
\end{remark} Here we just sketch out
the idea. Let $f\in L^2([0,T])$ and define the exponential
function of $f$ as $\varepsilon (f)=\exp( \int_{0}^{T}f(s)\,\mathrm{d}W(s))$.
Obviously $\varepsilon (f)\in L^2(\Omega)$ and has an exponential formula representation. Plug $F=\varepsilon(f)$ into both sides of (\ref{DysonSeries}). We obtain:
on one hand, since $\{\exp(\int_{0}^{u}f(s)\,\mathrm{d}W(u)-\frac{1}{2}\int_0^uf(s)^2\ud s)\}_{u\ge0}$ is a martingale, the
left hand-side of (\ref{DysonSeries}) is%
\begin{equation*}
E[\varepsilon (f)|\mathcal{F}_{t}]=\exp \Big( \int_{0}^{t}f(s)\,\mathrm{d}%
W(s)+\frac{1}{2}\int_{t}^{T}f(s)^{2}\,\mathrm{d}s\Big).
\end{equation*}%
On the other hand, we have
$$
\varepsilon (f)(\omega ^{t}) =\exp \Big( \int_{0}^{t}f(s)\,\mathrm{d}%
W(s)\Big)~\mbox{and}~D_{s}^{2}\varepsilon (f) =f(s)^{2}\varepsilon (f)\text{ for }s\in
[ 0,T].
$$
Thus the right hand-side of (\ref{DysonSeries}) is equal to
\begin{eqnarray*}
&&\sum_{i=0}^{\infty}\varepsilon (f)(\omega ^{t})\frac{1}{2^{i}i!}%
\int_{[t,T]^{i}}(f(s_{i})\ldots f(s_{1}))^{2}\,\mathrm{d}s_{i}\ldots \,%
\mathrm{d}s_{1}\\
&&=\exp \Big( \int_{0}^{t}f(s)\,\mathrm{d}W(s)\Big)\sum_{i=0}^{\infty} \frac{1}{i!}\Big(
\frac{1}{2}\int_{t}^{T}f(s)^{2}\,\mathrm{d}s\Big)^{i}\\
&&=\exp \Big( \int_{0}^{t}f(s)\,\mathrm{d}W(s)+\frac{1}{2}%
\int_{t}^{T}f(s)^{2}\,\mathrm{d}s\Big)\\
&&=E[\varepsilon (f)|\mathcal{F}_{t}].
\end{eqnarray*}%
Hence (\ref{DysonSeries}) holds for $F=\varepsilon (f)$. For general $F\in L^{2}(\Omega )$, the proof of (\ref{DysonSeries}) can be
completed by using the fact that the linear span of the exponential functions is a dense subset of $L^{2}(\Omega )$.
\section{Solutions of Some Problems by Using Dyson Series}
In this section we apply Dyson series expansion to compute $E[F|\mathcal F_t]$ for 4 different $F$. The first example is a very well-known
example, but it illustrates nicely the computation of Dyson series in case
the random variable $F$ (seen as a functional of Brownian motion) is not
path-dependent. In the second example, the functional $F$ is path-dependent.
The third example is path-dependent again and illustrates the problem of
convergence of the Dyson series. The fourth example shows a new
representation of the price of a discount bond in the Cox-Ingersoll-Ross
model (see \cite{Shreve} for a discussion of the problem).

\subsection{The Heat Kernel}
Consider the random variable
\begin{equation*}
F=f(W(T),T),
\end{equation*}
where $f(x,T)$ is some heat kernel satisfying the following differential equation
\begin{equation*}
\frac{\partial ^{2n}f}{\partial x^{2n}}=(-2)^n\frac{\partial ^{n}f}{\partial
T^{n}}.
\end{equation*}
For $t\le s_1,\ldots,s_i\le T$,
\begin{eqnarray*}
&&(D_{s_i}^2\ldots D_{s_1}^2f(W(T),T))(\omega^t)=\frac{\partial^{2i} f(x,T)}{\partial x^{2i}}\Bigg| _{x=W(T)}(\omega^t)\\
&&=(-2)^i\frac{\partial^i
f(W(T),T)}{\partial T^i}(\omega^t)=(-2)^i\frac{\partial^i f(W(t),T)}{%
\partial T^i}.
\end{eqnarray*}
It follows that
\begin{equation*}
\frac{1}{2^ii!}\int_{[t,T]^i}(D_{s_i}^2\ldots
D_{s_1}^2f(W(T),T))(\omega^t)\, \mathrm{d} s_i\ldots\, \mathrm{d} s_1=\frac{%
(t-T)^i}{i!}\frac{\partial^i f(W(t),T)}{\partial T^i}.
\end{equation*}
Then from (\ref{DysonSeries}) we see that, the Dyson series expansion of $E\left[ F|%
\mathcal{F}_{t}\right]$ coincides with the Taylor expansion of $%
t\mapsto f(x,t)$ around $t=T$, evaluated at $x=W(t)$:
$$
E\left[ F|\mathcal{F}_{t}\right] =\sum_{i=0}^{\infty}\frac{(t-T)^i}{i!}%
\frac{\partial^i f(W(t),T)}{\partial T^i}=f(W(t),t).
$$
As a particular case,
\begin{equation*}
f_\tau:~ (x,T) \longmapsto\frac{1}{\sqrt{\tau -T}}\exp \Big(-\frac{x^{2}}{%
2(\tau -T)}\Big),~\tau>T
\end{equation*}
is such a heat kernel. When taking
\begin{equation*}
F=f_{\tau}(W(T),T)=\frac{1}{\sqrt{\tau -T}}\exp \Big(-\frac{W(T)^{2}}{2(\tau -T)}\Big),
\end{equation*}
we get
\begin{equation*}
E\left[ F|\mathcal{F}_{t}\right]=\frac{1}{\sqrt{\tau -t}}\exp \Big(-\frac{%
W(t)^2}{2(\tau -t)}\Big).
\end{equation*}
\textbf{Observation: }We deliberately took $\tau >T$ so that $%
F$ would be well-defined and infinitely Malliavin differentiable. This example is not new, in the sense that it could have been obtained by
applying $\exp(\frac{1}{2} (T-t) \frac{\partial ^2}{\partial x^2})$, i.e.,
the time-evolution operator of the heat equation, to the function $f_\tau(
x,T)$ (see
\cite{Hunter}, Page 162). This example hints to the fact that our
time-evolution operator is an extension of the time-evolution operators
for the heat kernel, the latter being applicable to path-independent
problems, and the former being applicable to path-dependent problems.

\subsection{The Merton Interest Rate Model}

\noindent Let $F=\exp (-\int_{0}^{T}W(u)\, \mathrm{d} u)$. By It\^o's
formula we have
\begin{equation*}
\int_{0}^{T}W(u)\, \mathrm{d} u=\int_{0}^{T}(T-u)\, \mathrm{d} W(u).
\end{equation*}%
This implies that for $t\le s_1,\ldots,s_i\le T$,
\begin{equation*}
D_{s_i}^2\ldots D_{s_1}^2F=(T-s_i)^2\ldots(T-s_1)^2F.
\end{equation*}
We also observe that \begin{equation*}
F(\omega ^{t}) =\exp
\Big(-\int_{0}^{t}W(u)\, \mathrm{d} u-W(t)(T-t)\Big).
\end{equation*}
Therefore by (\ref{DysonSeries}), the Dyson series expansion of $E\left[ F|%
\mathcal{F}_{t}\right]$ is explicitly given as
$$
E\left[ F|\mathcal{F}_{t}\right] =F(\omega ^{t})\sum_{i=0}^{\infty }\frac{1%
}{i!}\Big( \frac{1}{2}\int_{t}^{T}(T-s)^{2}\, \mathrm{d} s\Big) ^{i}
$$
and it leads to
$$
E\left[ F|\mathcal{F}_{t}\right]=\exp \Big(-\int_{0}^{t}W(u)\, \mathrm{d} u-W(t)(T-t)+\frac{1}{6}%
(T-t)^{3}\Big).
$$

\subsection{Moment Generating Function of Geometric Brownian Motion}

Let $X(T) =e^{M-\sigma W(T)}$ ($M\in\mathbb R$, $\sigma>0$) be a geometric Brownian motion (the Brownian motion is scaled and noncentral) at time $T>0$. Recall that $X(T)$ is in fact a lognormal random variable. Denote by $F=e^{-X(T)}$, then $E[F]$ is the moment generating function of $X(T)$ evaluated at $-1$. In this example we obtain the Dyson series expansion of $E[F|\mathcal{F}_t]$ for $t\le T$ and compare it with the Taylor series expansion at $t=0$.

First, observe that $F$ is an infinitely continuously differentiable function of  $W(T)$. Then by definition of Malliavin derivative operator (see e.g. Definition 1.2.1 in \cite{r8}), for $t\le s_{1},\ldots,s_n\le T$,
$$
D_{s_{n}}^{2}\ldots D_{s_{1}}^{2}F=D_T^{2n}F.
$$
Note that by induction, we can easily show
\begin{equation*}
D_T^{2n}F=F\sigma ^{2n}\sum_{i=0}^{2n}(-1)^{i}%
\begin{Bmatrix}
2n \\
i%
\end{Bmatrix}%
e^{i(M-\sigma W(T))}~\mbox{for}~n\ge0,
\end{equation*}%
where $%
\begin{Bmatrix}
j \\
i%
\end{Bmatrix}%
$ for $i\le j$ denote the Stirling numbers of the second kind, with convention $
\begin{Bmatrix}
0 \\
0%
\end{Bmatrix}%
$ $=1$, and $%
\begin{Bmatrix}
j \\
0%
\end{Bmatrix}%
$ $=0$ for any $j> 0$. Therefore for $t\le T$,%
\begin{equation*}
\frac{1}{2^{n}n!} \int_{[t,T]^n} (D_{s_{n}}^{2}\ldots D_{s_{1}}^{2}F)(\omega
^{t})\, \mathrm{d} s_{n}\ldots\, \mathrm{d} s_{1} =e^{-e^{M-\sigma W(t)}}%
\frac{(T-t) ^{n}\sigma ^{2n}}{2^{n}n!}\sum_{i=0}^{2n}(-1)^{i}%
\begin{Bmatrix}
2n \\
i%
\end{Bmatrix}%
e^{i(M-\sigma W(t))},
\end{equation*}
and by (\ref{DysonSeries}), the Dyson series expansion of $E[F|\mathcal{F}_{t}]$ is given as
\begin{equation}  \label{ourseries}
e^{-e^{M-\sigma W(t)}}\sum_{n=0}^{\infty
}\sum_{i=0}^{2n}\frac{(\frac{\sigma ^{2}(T-t) }{2})^{n}(-1)^{i}}{n!}%
\begin{Bmatrix}
2n \\
i%
\end{Bmatrix}%
e^{i(M-\sigma W(t))}.
\end{equation}
It is known that the Laplace transform of the lognormal distribution does not have closed-form (see \cite{Heyde}) nor convergent series representation (see e.g. \cite{Holgate}). In particular, it is shown in \cite{Tellam} that the corresponding Taylor series is divergent. However divergence does not mean "useless". A number of alternative divergent series can be applied for numerical computation purpose. The study on the approximations of Laplace transforms of lognormal distribution (as well as its characteristic functions) has been long standing and many methods by using divergent series have been developed. For this area we refer to \cite{Barakat,Holgate,Asmussen} and the references therein.  In this example, we obtain the Dyson series (\ref{ourseries}) as a new divergent series representation of the Laplace transform of $X(T)$. Next we compare our Dyson series expansion of $E[F]$ (taking $M=0$, $t=0$ in (\ref{ourseries})):
\begin{equation}
\label{Dyson1}
\sum_{n=0}^{\infty
}\sum_{i=0}^{2n}\frac{(\sigma ^{2}T)^{n}(-1)^{i}}{2^nn!e}%
\begin{Bmatrix}
2n \\
i%
\end{Bmatrix}
\end{equation} to its Taylor series expansion
\begin{equation}
\sum_{n=0}^{\infty }\frac{(-1)^{n}}{n!}e^{\frac{n^{2}\sigma ^{2}T}{2}}
\label{Taylor}
\end{equation}%
by using numerical methods. Namely, in a particular parameter setting $(T,\sigma)=(1,0.6)$, we compute the first 40 partial sums of Dyson series and the first 10 partial sums of Taylor series of $E[F]$. The results are presented below as illustrations and tables.
\begin{figure}[H]
\begin{minipage}[b]{0.45\linewidth}
\centering
\includegraphics[width=\textwidth]{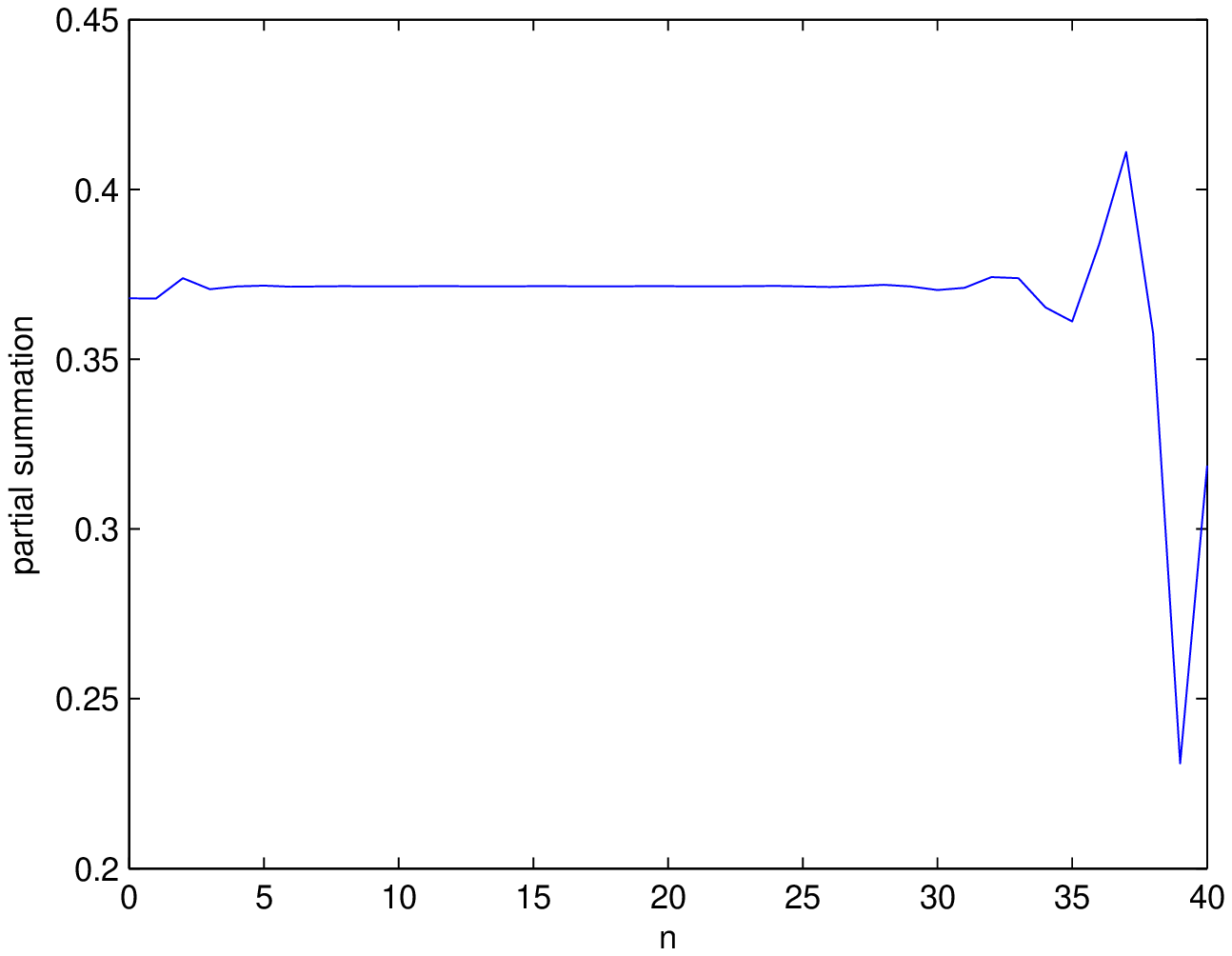}
\caption{The first $40$ partial sums of Dyson series with $(M,T,\sigma)=(0,1,0.6)$.}
\end{minipage}
\hspace{0.6cm}
\begin{minipage}[b]{0.45\linewidth}
\centering
\includegraphics[width=\textwidth]{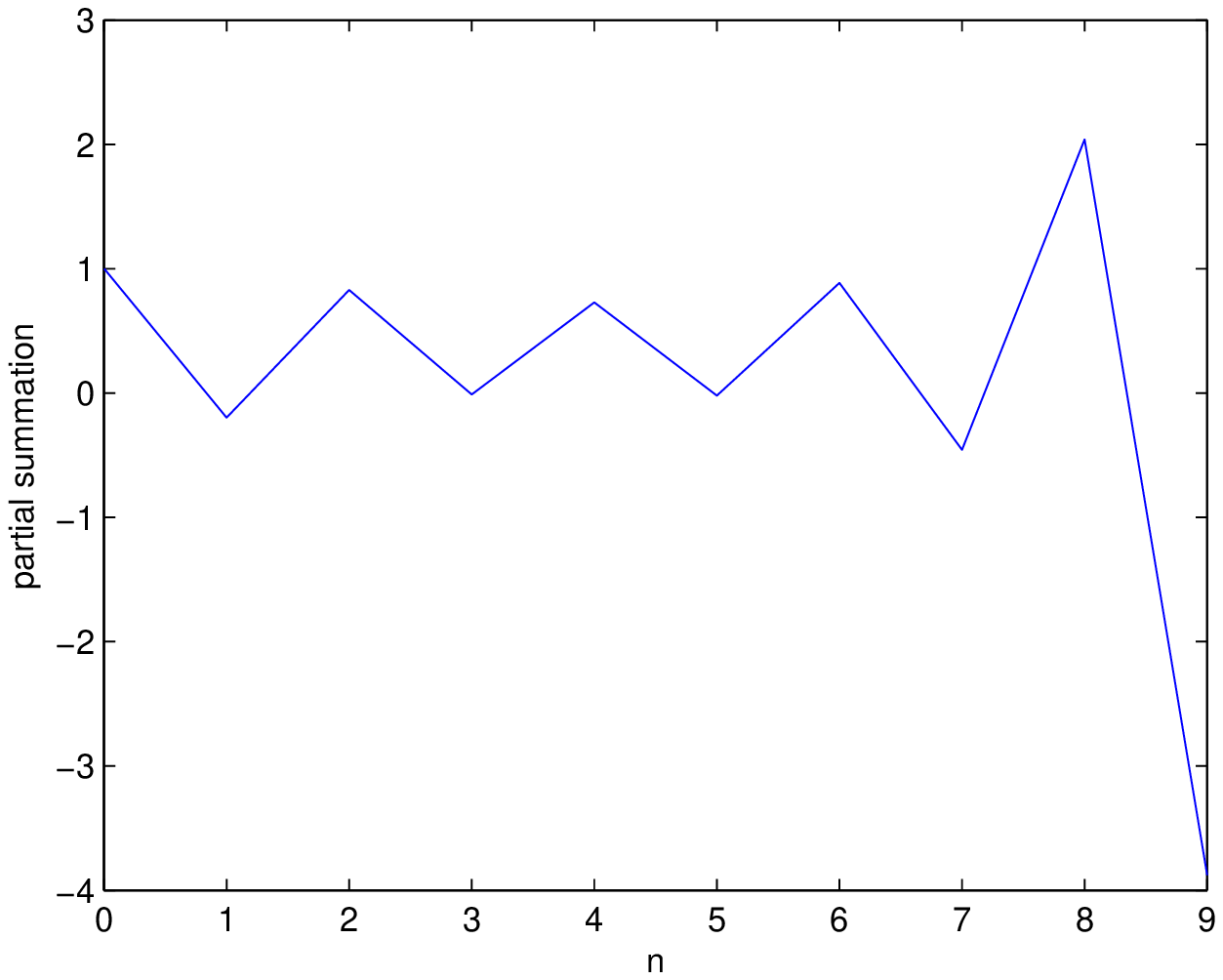}
\caption{The first $10$ partial sums of Taylor series with $(M,T,\sigma)=(0,1,0.6)$.}
\end{minipage}
\end{figure}
\begin{table}[H]
\begin{center}
\begin{tabular}{|c|c|c|}
\hline
$n$& Sum of first $n$ terms (Dyson series) & Sum of first $n$ terms (Taylor series)  \\
 \hline
0&0.3679 & 1  \\ \hline
1&0.3679 & -0.1972  \\ \hline
2&0.3679 & 0.8300  \\ \hline
3&0.3679 & -0.0122  \\ \hline
4&0.3738 & 0.7301  \\ \hline
5&0.3706 & -0.0201  \\ \hline
6&0.3706 & 0.8855  \\ \hline
7&0.3714 & -0.4575  \\ \hline
8&0.3714 & 2.0403  \\ \hline
9&0.3717 & -3.8787  \\ \hline\hline
True value&0.3717 & 0.3717  \\ \hline
\end{tabular}%
\end{center}
\caption{Values of the first 10 partial sums of Dyson series and Taylor series.}
\end{table}
From the results we see that interesting phenomena arise:
\begin{description}
  \item[(1)] The Taylor series (\ref{Taylor}) diverges
much earlier (as $n$ increases) and has overall much larger deviation than the Dyson series (\ref{Dyson1}). From Figures 1-2, we observe that the summation of the first terms of series (\ref{ourseries}) starts to diverge at around $n=35$ and the signal amplitude is less than $0.02$, while the Taylor expansion diverges at $n=8$ with a jump larger than $2$.
  \item[(2)] The first terms of Dyson series are good approximations of the Laplace transform, while the Taylor series' are not. In Table 1, we compare the first 10 partial sums of the series to the "true value" of $E[F]=\exp(e^{-0.6W(1)})$ by Monte Carlo simulation: 0.3717. The latter value is obtained through generating the lognormal variable $2^{20}$ times. We see that the first 10 partial sums of Dyson series have bias less than 0.0038, while the first 10 partial sums of Taylor series are not good estimates at all. Heuristically speaking, Dyson series representations perform much better on approximation because the exponential formula of $E[F|\mathcal F_t]$ is an expansion around $F(\omega^t)$, but the Taylor series is based on an expansion around $0$ (its first term in the series is always $1$).
\end{description}
  In conclusion, numerical experiments show the Dyson series (\ref{ourseries}) provides new and good approximations of the Laplace transforms of lognormal distribution.

Also it is interesting to observe that (\ref{ourseries}) resembles the moment generating function for
a Poisson random variable $N$, which can be expressed as:
\begin{equation*}
E[\exp (zN)]=\sum_{n=0}^{\infty }\sum_{i=0}^{n} \frac{z^n}{n!}
\begin{Bmatrix}
n \\
i%
\end{Bmatrix}
\lambda^i,
\end{equation*}
where $N$ is a Poisson random variable with mean parameter $\lambda$.

\subsection{Bond Pricing in the Extended Cox-Ingersoll-Ross Model}

Assume that the interest rate is given by
\begin{equation}
\label{CIR}
\, \mathrm{d} r(s)=\left( -2b(s)r(s)+d\sigma (s)^{2}\right) \, \mathrm{d}
s+2\sigma (s)\sqrt{r(s)}\, \mathrm{d} W(s),
\end{equation}%
where $r(0)=r_{0}$, $b(s)$ and $\sigma (s)$ are deterministic functions and $d$ is a positive integer.

By It\^o's lemma and L\'evy's theorem (see e.g.\cite{Shreve}) the interest
rate $\{r(s)\}_{s\ge0}$ can be represented as
\begin{equation}  \label{rel}
\{r(s)\}_{s\ge0}=\left\{\sum_{i=1}^{d}X_{i}(s)^2\right\}_{s\ge0},
\end{equation}%
where $X_{i}$ is the Ornstein-Uhlenbeck process defined by
\begin{equation*}
\, \mathrm{d} X_{i}(s)=-b(s)X_{i}(s)\, \mathrm{d} s+\sigma (s)\, \mathrm{d}
W_{i}(s),
\end{equation*}%
with $X_{i}(0)=\sqrt{\frac{r_{0}}{d}}$ for all $i=1,\ldots,d$ and $W_i\in(\Omega,\{\mathcal F_t^i\}_{t\ge0},P)$ being independent standard Brownian motions. Let
\begin{equation*}
F=\exp \Big( -\int_{t}^{T}r(s)\, \mathrm{d} s\Big) .
\end{equation*}%
Our goal is to find the bond price $E[F|\mathcal{F}_{t}]$ for $t\in[0,T]$. Since $F$ can be
written as $F=\prod_{i=1}^{d}F_{i}$ where $F_{i}:=\exp (
-\int_{t}^{T}X_{i}(s)^{2}\, \mathrm{d} s) $, then we can decompose $E[F|\mathcal{F}_{t}]$ into product of independent
conditional expectations:
\begin{equation}
E[F|\mathcal{F}_{t}]=\prod_{i=1}^{d}E[F_{i}|\mathcal{F}_{t}^{i}].
\label{decomCIR}
\end{equation}
Below we compute $E[F_{i}|\mathcal{F}_{t}^{i}]$ for each $i$ to obtain $E[F|\mathcal{F}_{t}]$ by applying (\ref{DysonSeries}). We remark that, for each $E[F_{i}|\mathcal{F}_{t}^{i}]$,  acting the operator $\omega^t$ means to freeze its corresponding $W_i$. Two cases are discussed according to whether $b$ is time-dependent or not.
\begin{description}
\item[Case 1] We first consider the simple case $b\equiv 0$.
\end{description}
From It\^o formula we see
\begin{equation}
X_{i}(s)=X_{i}(0)+\int_{0}^{s}\sigma (u)\, \mathrm{d} W_i (u).
\end{equation}
For each $i=1,\ldots,d$, take $s\in [t,T],u\geq s$. Then by Remark \ref{freezing}, we obtain
\begin{eqnarray}
\label{computefreezing}
&&X_i(s)^2(\omega ^{t}) = (X_i(s)(\omega ^{t}))^2=X_i(t)^2;\nonumber \\
&&(D_{s}X_i(u)^2)(\omega ^{t}) =2\sigma (s)X_i(t);\nonumber \\
&&(D_{s}^{2}X_i(u)^2)(\omega ^{t}) =2\sigma ^{2}(s).
\end{eqnarray}
By (\ref{decomCIR}), (\ref{DysonSeries}) and (\ref{computefreezing}), the first terms of the
Dyson series are explicitly given as
\begin{eqnarray}  \label{CIRcompu2}
&&E[F|\mathcal{F}_{t}]=e^{-(T-t)r(t)}\prod_{i=1}^{d} \bigg\{ 1+\frac{1}{2}%
\int_{t}^{T}\sigma ^{2}(s_{1})\big(4(T-s_{1})^{2}X_{i}(t)^{2}-2(T-s_{1})\big)\,
\mathrm{d} s_{1}  \notag \\
&&~~+\frac{1}{2^{2}}\int_{t}^{T}\int_{s_{1}}^{T}\sigma ^{2}(s_{1})\sigma
^{2}(s_{2}) \Big\{ 16(T-s_{1})^{2}(T-s_{2})^{2}X_{i}(t)^{4}  \notag \\
&&~~-\left( 8(T-s_{1})^{2}(T-s_{2})+40(T-s_{1})(T-s_{2})^{2}\right)
X_{i}(t)^{2}+8(T-s_{2})^{2}  \notag \\
&&~~+4(T-s_{1})(T-s_{2})\Big\}\, \mathrm{d} s_{1}\, \mathrm{d} s_{2}+\ldots%
\bigg \}.
\end{eqnarray}
Let us denote by $A_0,A_1,A_2$ the coefficients of $%
1,X_i(t)^2,X_i(t)^4$ in the expansion (\ref{CIRcompu2}), respectively. Then (\ref{CIRcompu2}) can be represented by
\begin{eqnarray}  \label{CIR_Sixian_calc}
E[F|\mathcal{F}_{t}]=e^{-(T-t)r(t)}
(A_0^d+A_0^{d-1}A_1r(t)+A_0^{d-1}A_2r(t)^2+\ldots).
\end{eqnarray}
To be more explicit we denote $r_{n-m}^{(m)}$ as the remainder of the series expansion of $A_m$ ($n$ corresponds the multiplicity of integral $\int_t^T\ldots\int_{s_{n-1}}^T$ of $r_{n-m}^{(m)}$'s first term) and can write
\begin{eqnarray*}
A_0&=&1-\int\limits_{t}^{T}(T-s_{1})\sigma(s_{1})^2\, \mathrm{d}
s_{1}\\
&&~~+\int_{t}^{T}\int_{s_{1}}^{T}\big(2(T-s_{2})^{2}+(T-s_{1})(T-s_{2})\big)%
\sigma(s_{1})^2\sigma(s_{2})^2\, \mathrm{d} s_{2}\, \mathrm{d}
s_{1}+r_3^{(0)}; \\
A_1&=&\int_t^T 2(T-s_1)^2\sigma(s_1)^2\, \mathrm{d} s_1 \\
&&~~+\int_t^T \int_{s_1}^T
\big(10(T-s_1)(T-s_2)^2-2(T-s_1)^2(T-s_2)\big)\sigma(s_1)^2\sigma(s_2)^2\, \mathrm{d}
s_2\, \mathrm{d} s_1+r_2^{(1)}; \\
A_2&=&\int_t^T \int_{s_1}^T 4(T-s_1)^2(T-s_2)^2\sigma(s_1)^2\sigma(s_2)^2\,
\mathrm{d} s_2\, \mathrm{d} s_1+r_1^{(2)}.
\end{eqnarray*}
\textbf{Remark}: It is worth noting that the explicit forms of $A_0,A_1,A_2...$ remain unknown (it is notoriously difficult to compute the remaining terms $r_{n-m}^{(m)}$ and this subject of calculus is still quite open). As a consequence, it is not sure whether $A_0^d,A_0^{d-1}A_1,A_0^{d-1}A_2$ are the coefficients of $1,r(t),r(t)^2$ in the Taylor expansion of $e^{(T-t)r(t)}E[F|\mathcal F_t]$ around $r(t)$, namely we can not show theoretically $A_0A_n=\frac{1}{n!}A_1^n$ for all $n\geq 2$. However, if we suppose $\sup_{s\in [t,T]}\sigma(s)=\sigma$ exists and $\sigma^2 \tau<\frac{1}{2}$ where $\tau:=T-t$, it is not hard to show by induction that for all positive integer $m$ and $n$, $r^{(m)}_{n-m}$ can be bounded by $c\tau^{m+2n}$, $c>0$ is a constant which does not depend on $m$ and $n$. Moreover, we can interestingly check our coefficients of first terms with some known results. For example, in the particular case $\sigma\equiv1$, $d=1$, there is
an existing analytical formula (see \cite{Shreve} again):
\begin{equation}
E[F|\mathcal{F}_{t}]=(\text{sech}( \sqrt{2}\tau )) ^{%
\frac{1}{2}}e^{-\frac{r(t)}{\sqrt{2}}\tanh \sqrt{2}\tau },  \label{formula}
\end{equation}
where $\text{sech}(\cdot)$, $\tanh(\cdot)$ denote  hyperbolic secant and hyperbolic tangent respectively. Applying Taylor expansion around $\tau=0$ in (\ref{formula}) leads to
\begin{eqnarray}  \label{CIRbond}
&&E[F|\mathcal{F}_{t}]=e^{-(T-t)r(t)} \bigg\{ \big( 1-\frac{1}{2}\tau ^{2}+\frac{7}{24}\tau
^{4}+O(\tau^6)\big)\nonumber\\
&&~~ +r(t)\big( \frac{2}{3}\tau ^{3}-\frac{8}{15}\tau
^{5}+O(\tau^7)\big) +r(t)^{2}\big( \frac{2}{9}\tau ^{6}+O(\tau^8)\big)
+\ldots\bigg\},
\end{eqnarray}
By taking $\sigma\equiv1$, $d=1$ in (\ref{CIRcompu2}), we see that our first terms agree with those in (\ref{CIRbond}).
\begin{description}
\item[Case 2] We now consider a more general case, where $b$ is a non-zero deterministic function.
\end{description}
Again, the problem of obtaining the general term in the series (\ref{DysonSeries}) for $E[F|\mathcal F_t]$ is still open. Instead we compute the first 2 terms of its Taylor expansion around $r(t)$. Denote by $\tilde b(s):=\int_0^s b(u)\mathrm{d}u$, then by It\^o formula, for each $i=1,\ldots,d$,
\begin{equation*}
X_{i}(s)=X_{i}(0)e^{-\tilde b(s)}+e^{-\tilde b(s)}\int_{0}^{s}e^{\tilde b(v)}\sigma (v)\, \mathrm{d}
W_{i}(v).
\end{equation*}%

By using a similar computation as in (\ref{CIR_Sixian_calc}),
\begin{equation}  \label{Dysongeneral}
E[F|\mathcal{F}_{t}]=\exp \Big( \Big(-\int_t^T e^{2\tilde b(t)-2\tilde b(u)}\mathrm{d} u\Big)  r(t)\Big) (A_0(\tilde b)^d+A_0(\tilde b)^{d-1}A_1(\tilde b)r(t)+\ldots),
\end{equation}
with
\begin{eqnarray*}
A_0(\tilde b)&=&1-\int_t^T\int_s^T e^{2\tilde b(s)-2\tilde b(u)}\sigma (s)^{2}\, \mathrm{d} u \mathrm{d} s+\ldots; \\
A_1(\tilde b)&=&2e^{2\tilde b(t)}\int_t^T\Big( \int_s^T e^{-2\tilde b(u)}\mathrm{d} u\Big) ^{2}e^{2\tilde b(s)}\sigma(s)^2 \, \mathrm{d} s+\ldots.
\end{eqnarray*}
Now let's introduce an application of (\ref{Dysongeneral}) to some pricing problem.
Recall that (see \cite{Shreve}) the bond price is
affine with respect to $r(t)$ and satisfies
\begin{equation}
E[F|\mathcal{F}_{t}]=\exp (-C(t,T)r(t)-A(t,T)),  \label{generalcase}
\end{equation}
where $C(t,T)$ solves the time-dependent Riccati equation below:
\begin{equation}
\frac{\partial C(t,T)}{\partial t}=2b(t)C(t,T)+2\sigma (t)^{2}C(t,T)^2-1,  \label{Rica}
\end{equation}%
and $A$ satisfies $\frac{\partial A(t,T)}{\partial t}=-d\sigma(t)^2C(t,T)$. By (\ref{Dysongeneral}) and the Taylor expansion of the right-hand side of (\ref{generalcase}), we have:
\begin{eqnarray*}
&&A_0(\tilde b)^d+A_0(\tilde b)^{d-1}A_1(\tilde b)r(t)+\ldots \\
&&=\exp \Big( \Big(\int_t^T e^{2\tilde b(t)-2\tilde b(u)}\mathrm{d} u-C(t,T)\Big)  r(t)-A(t,T)\Big) \\
&&=e^{-A(t,T)}\Big(1+\Big( \int_t^T e^{2\tilde b(t)-2\tilde b(u)}\mathrm{d} u-C(t,T)\Big)  r(t)+\ldots\Big).
\end{eqnarray*}
The above equation allows us to obtain the solution of the Riccati equation (\ref{Rica}) as
\begin{equation*}
C(t,T)=-\frac{A_1(\tilde b)}{A_0(\tilde b)}+\int_t^T e^{2\tilde b(t)-2\tilde b(u)}\mathrm{d} u.
\end{equation*}
In the meanwhile
\begin{equation*}
A(t,T)=-d\log{A_0(\tilde b)}.
\end{equation*}

\section{Conclusion and Future Work}

For future work, we intend to design and analyze new numerical schemes that
implement the Dyson series to solve BSDEs. The main weakness of Theorem \ref%
{ExpFormula} is that it currently requires the functional $F$ to be
infinitely Malliavin differentiable. It is unknown at this point whether
this smoothness requirement can be relaxed. Theorem \ref{ExpFormula} can certainly be
extended to a filtration generated by several Brownian motions, and probably
to L\'evy processes. A generalization from representation of martingales to
representation of semi-martingales would also be interesting.

\section{Appendix}

\subsection{Proof of Theorem \ref{BTE}}
Fix $m\in\{0,1,\ldots,M-1\}$, we denote by $t=m\Delta$ and $T=(m+1)\Delta$ for simplifying notations. We remind the reader of
Proposition 1.2.4 in \cite{r8}, namely, if $F\in \mathbb{D}^{1,2}$ \footnote{The definition of this space is given on Page 26 in \cite{r8}. It is sufficient to note that in this work $F$ belongs to $\mathbb D_{\infty}$, a subspace of $\mathbb D^{1,2}$.}, then for $t\leq s$,
\begin{equation}
\label{krompir2}
D_{t}(E[F|\mathcal{F}_{s}])=E[D_{t}F|\mathcal{F}_{s}].
\end{equation}
Using (\ref{krompir2})\ and the Clark-Ocone formula (see, e.g. Proposition 1.3.5 in \cite{r8}),
we get, for any integer $l\ge0$:
\begin{eqnarray}
E[D_{T}^{l}F|\mathcal{F}_{t}]
&=&E[D_{T}^{l}F]+\int_{0}^{t}E[D_{s}E[D_{T}^{l}F|\mathcal{F}_{t}]|\mathcal{F}%
_{s}]\, \mathrm{d} W(s)  \notag \\
&=&E[D_{T}^{l}F]+\int_{0}^{t}E[D_{s}D_{T}^{l}F|\mathcal{F}_{s}]\, \mathrm{d}
W(s)  \notag \\
&=&E[D_{T}^{l}F]+\int_{0}^{T}E[D_{s}D_{T}^{l}F|\mathcal{F}_{s}]\, \mathrm{d}
W(s)  \notag \\
&&-\int_{t}^{T}E[D_{s}D_{T}^{l}F|\mathcal{F}_{s}]\, \mathrm{d} W(s)  \notag
\\
&=&E[D_{T}^{l}F|\mathcal{F}_{T}]-\int_{t}^{T}E[D_{s}D_{T}^{l}F|\mathcal{F}%
_{s}]\, \mathrm{d} W(s).  \label{trucat}
\end{eqnarray}
Since $F$ is assumed to be cylindrical, then by the definition of Malliavin derivative operator,  we have
\begin{equation}
\label{const}
D_{s}D_{T}^{l}F=D_{T}^{l+1}F,~\mbox{for}~s\in (t,T].
\end{equation}
It results from (\ref{trucat}) and (\ref{const}) that
\begin{equation*}  \label{tugudu}
E[D_{T}^{l}F|\mathcal{F}_{t}]=E[D_{T}^{l}F|\mathcal{F}_{T}]-
\int_{t}^{T}E[D_{T}^{l+1}F|\mathcal{F}_{s}]\, \mathrm{d} W(s).
\end{equation*}
We thus obtain
\begin{eqnarray*}
E[F|\mathcal{F}_{t}]&=&E[F|\mathcal{F}_{T}]-\int_{t}^{T}E[D_{T}F|\mathcal{F}%
_{s_{1}}]\, \mathrm{d} W(s_{1}) \\
&=&E[F|\mathcal{F}_{T}]-\int_{t}^{T}E[D_{T}F|\mathcal{F}_{T}]\delta
W(s_{1})\nonumber\\
&&+\int_{t}^{T}\int_{s_{1}}^{T}E[D_{T}^{2}F|\mathcal{F}_{s_{2}}]\delta
W(s_{2})\delta W(s_{1}).
\end{eqnarray*}%
We continue the expansion
iteratively:
\begin{eqnarray}
\label{DBTE1}
&&E[F|\mathcal{F}_{t}]=E[F|\mathcal{F}_{T}]-\int_{t}^{T}E[D_{T}F|\mathcal{F}%
_{T}]\delta W(s_{1})+\ldots  \nonumber \\
&&~~+(-1)^{L-1}\int_{t}^{T}\int_{s_{1}}^{T}\ldots%
\int_{s_{L-2}}^{T}E[D_{T}^{L-1}F|\mathcal{F}_{T}]\delta W(s_{L-1})\ldots\delta
W(s_{1})+R_{[t,T]}^{L},
\end{eqnarray}
where the remainder $R_{[t,T]}^{L}$ is given as
\begin{equation*}
R_{[t,T]}^{L}:=(-1)^L\int_{t}^{T}\int_{s_{1}}^{T}\ldots%
\int_{s_{L-1}}^{T}E[D_{T}^{L}F|\mathcal{F}_{s_{L}}]\delta
W(s_{L})\ldots\delta W(s_{1}).  \label{bound}
\end{equation*}
Note that the partial sum (\ref{DBTE1}) converges in $L^2(\Omega)$ as $L\rightarrow \infty$ is equivalent to $\|R_{[t,T]}^{L}\|_{L^2(\Omega)}\rightarrow0$ as $L\rightarrow\infty$. We then introduce two useful lemmas.
For simplifying notations, we denote the multiple Skorohod integral of a
 stochastic process $\{X(s)\}_{s\in\mathbb R^L}$ by $\delta^L(X)$:
\begin{equation*}
\delta ^{L}(X):=\int_{\mathbb{R}^{L}}X(s_{1},\ldots,s_{L})\left( \delta
W(s)\right) ^{\otimes L},
\end{equation*}%
where $\left( \delta W(s)\right) ^{\otimes L}:=\delta W(s_{L})\ldots \delta
W(s_{1})$ and similarly we denote $(\, \mathrm{d} s)^{\otimes {L}%
}:=\, \mathrm{d} s_{L}\ldots \, \mathrm{d} s_{1}$ in the sequel. The following lemma represents $E[\delta^L(X)^2]$.
\begin{lemma}
\label{Edel} For a stochastic process $X$ indexed by $\mathbb R^L$, its multiple Skorohod integral $\delta^L(X)$ is well-defined if $E[\delta ^{L}(X)^{2}]<\infty$. Moreover, we have
\begin{eqnarray*}
&&E[\delta ^{L}(X)^{2}]=\sum_{i=0}^{L}{\binom{L}{i}}^{2}i!E\Big[\int_{\mathbb R^i}\Big(\int_{\mathbb{R}%
^{2L-2i}}D_{x_{1},\ldots,x_{L-i}}^{L-i}X(s_{1},\ldots
,s_{L-i},s_{L-i+1},\ldots ,s_{L}) \\
&&~~\times D_{s_{1},\ldots ,s_{L-i}}^{L-i}X(x_{1},\ldots
,x_{L-i},s_{L-i+1},\ldots ,s_{L}) (\mathrm{d}x)^{\otimes (L-i)}(\mathrm{d}s)^{\otimes (L-i)}\Big)\mathrm{%
d}s_{L-i+1}\ldots\mathrm{d}s_{L}\Big].
\end{eqnarray*}
\end{lemma}
This lemma is given by (2.12) in \cite{Ivan}, where the proof is not provided. We came up with a proof, which is
available upon request.
\begin{lemma}
\label{Remainder} We have
\begin{equation*}
E\left[( R_{[t,T]}^{L}) ^{2}\right]\leq \sum_{i=0}^{L}E\left[
( D_{T}^{2L-i}F) ^{2}\right] {\binom{L}{i}}^{4}\frac{i!}{\left(
L!\right) ^{2}}(T-t)^{2L-i}\xrightarrow[L\rightarrow \infty]{}0.  \label{ERL}
\end{equation*}
\end{lemma}
\begin{proof}
The proof is based on finding a fine upper bound of the remainder $R_{[t,T]}^{L}$'s $L^2(\Omega)$ norm. We first observe that
\begin{equation}
\label{DefR}
R_{[t,T]}^{L}=\frac{(-1)^{L}}{L!}\int_{\mathbb{R}^{L}}H_{L}(s_{1},\ldots,s_{L})%
\left( \delta W(s)\right) ^{\otimes L},
\end{equation}%
where
\begin{equation}
H_{L}(s_{1},\ldots,s_{L}):=\sum_{\sigma \in S_{L}}E[D_{T}^{L}F|\mathcal{F}%
_{s_{\sigma (L)}}]\chi _{[t\leq s_{\sigma (1)}\leq \ldots\leq s_{\sigma (L)}\leq
T]}(s_{1},\ldots,s_{L})  \label{defH}
\end{equation}%
is a symmetric function with $S_{L}$ being the collection of all permutations on $\{1,\ldots,L\}$. Then according to
Lemma \ref{Edel} and Fubini's theorem, we obtain
\begin{eqnarray}
 \label{Edelta}
&&E[\delta ^{L}(H_{L})^{2}] \notag \\
&&=\sum_{i=0}^{L}{\binom{L}{i}}^{2}i!E\Big[\int_{\mathbb R^i}\Big(\int_{\mathbb{R}^{2L-2i}}
D_{x_{1},\ldots,x_{L-i}}^{L-i}H_{L}(s_{1},\ldots,s_{L})  \notag \\
&&~~\times
D_{s_{1},\ldots,s_{L-i}}^{L-i}H_{L}(r_{1},\ldots,r_{L})
(\,%
\mathrm{d}x)^{\otimes (L-i)}(\mathrm{d}s)^{\otimes (L-i)}\Big)\mathrm{d}s_{L-i+1}\ldots\mathrm{d}s_{L}\Big],
\end{eqnarray}%
with $r_{l}:=x_{l}\chi _{\{l\leq L-i\}}+s_{l}\chi _{\{l>L-i\}}$, for $l=1,\ldots,L$. By definition of $H_L$ in (\ref{defH}),
\begin{eqnarray*}
&&D_{x_{1},\ldots,x_{L-i}}^{L-i}H_{L}(s_{1},\ldots,s_{L}) \\
&&=\sum_{\sigma \in S_{L}}E[D_{x_{1},\ldots,x_{L-i}}^{L-i}D_{T}^{L}F|\mathcal{F}%
_{s_{\sigma (L)}}]\chi _{[ t,s_{\sigma (L)}]^{L-i}}(x_{1},\ldots,x_{L-i})\chi
_{\{t\leq s_{\sigma (1)}\leq\ldots\leq s_{\sigma (L)}\leq T\}}(s_{1},\ldots,s_{L}).
\end{eqnarray*}
Similarly, we also have
\begin{eqnarray}
\label{DH}
&&D_{s_{1},\ldots,s_{L-i}}^{L-i}H_{L}(r_{1},\ldots,r_{L})\nonumber \\
&&=\sum_{\sigma ^{\prime }\in S_{L}}E[D_{s_{1},\ldots,s_{L-i}}^{L-i}D_{T}^{L}F|%
\mathcal{F}_{r_{\sigma' (L)}}]\chi _{[ t,r_{\sigma'
(L)}]^{L-i}}(s_{1},\ldots,s_{L-i})\chi _{\{t\leq r_{\sigma ^{\prime }(1)}\leq \ldots\leq
r_{\sigma ^{\prime }(L)}\leq T\}}(r_{1},\ldots,r_{L}).\nonumber\\
\end{eqnarray}
 It follows from (\ref{DH}), Cauchy-Schwarz inequality, the fact that $F$ is cylindrical and the inequality $E[(E[F|\mathcal F_t])^2]\le E[F^2]$ that
\begin{eqnarray}
\label{ED}
&&E\Big[
D_{x_{1},\ldots,x_{L-i}}^{L-i}H_{L}(s_{1},\ldots,s_{L})D_{s_{1},\ldots,s_{L-i}}^{L-i}H_{L}(r_{1},\ldots,r_{L})%
\Big]   \notag   \\
&&=E\Big[ \sum_{\sigma \in S_{L}}\sum_{\sigma ^{\prime }\in S_{L}}E[D_{x_{1},\ldots,x_{L-i}}^{L-i}D_{T}^{L}F|\mathcal{F}_{s_{\sigma
(L)}}]E[D_{s_{1},\ldots,s_{L-i}}^{L-i}D_{T}^{L}F|\mathcal{F}_{r_{\sigma
^{\prime }(L)}}]\Big]   \notag \\
&&~~\times\chi _{[ t,s_{\sigma (L)}]^{L-i}}(x_{1},\ldots,x_{L-i})\chi _{\{t\leq
s_{\sigma (1)}\leq \ldots\leq s_{\sigma (L)}\leq T\}}(s_{1},\ldots,s_{L})  \notag
\\
&&~~\times\chi _{[t,r_{\sigma ^{\prime }(L)}]^{L-i}}(s_{1},\ldots,s_{L-i})\chi
_{\{t\leq r_{\sigma ^{\prime }(1)}\leq \ldots\leq r_{\sigma ^{\prime }(L)}\leq
T\}}(r_{1},\ldots,r_{L})  \notag \\
&&\le E\left[ ( D_{T}^{2L-i}F) ^{2}\right]f(x_{1},\ldots,x_{L-i},s_{1},\ldots,s_{L}),
\end{eqnarray}%
where the function
\begin{eqnarray*}
&&f(x_{1},\ldots,x_{L-i},s_{1},\ldots,s_{L}):=\sum_{\sigma \in S_{L}}\sum_{\sigma ^{\prime }\in S_{L}}\chi _{[t,T]^{2L-2i}}(x_{1},\ldots,x_{L-i},s_{1},\ldots,s_{L-i}) \\
&&~~\times\chi _{\{t\leq s_{\sigma (1)}\leq \ldots\leq s_{\sigma (L)}\leq
T\}}(s_{1},\ldots,s_{L})\chi _{\{t\leq r_{\sigma ^{\prime }(1)}\leq\ldots\leq
r_{\sigma ^{\prime }(L)}\leq T\}}(r_{1},\ldots,r_{L})
\end{eqnarray*}%
is symmetric among each of the three groups of variables: $\{x_{1},\ldots,x_{L-i}\},%
\{s_{1},\ldots,s_{L-i}\}$ and $\{s_{L-i+1},\ldots,s_{L}\}$. We thus obtain%
\begin{eqnarray}
\label{f}
&&\int_{[t,T]^{2L-i}}f(x_{1},\ldots,x_{L-i},s_{1},\ldots,s_{L})(\,%
\mathrm{d}x)^{\otimes (L-i)}(\mathrm{d}s)^{\otimes L} \nonumber\\
&&=\left( (L-i)!\right)
^{2}i!\int_{\mathcal D}f(x_{1},\ldots,x_{L-i},s_{1},\ldots,s_{L})(\,%
\mathrm{d}x)^{\otimes (L-i)}(\mathrm{d}s)^{\otimes L},
\end{eqnarray}%
with $\mathcal D=\{(x_1,\ldots,x_{L-i},s_{1},\ldots,s_L)\in[t,T]^{2L-i}:x_{1}\leq \ldots\leq x_{L-i};\
s_{1}\leq\ldots\leq s_{L-i};\ s_{L-i+1}\leq \ldots\leq s_{L}\}$. Recall that by
basic calculation the following property holds:
\begin{equation}
\label{bound-f}
\int_{t\leq x_{1}\leq \ldots\leq x_{n}\leq T}(\,\mathrm{d}x)^{\otimes n}\leq
\int_{t\leq x_{1}\leq \ldots\leq x_{n-i}\leq T,t\leq
x_{n-i+1}\leq \ldots\leq x_{n}\leq T}(\,\mathrm{d}x)^{\otimes n}.
\end{equation}%
It results from (\ref{Edelta}), (\ref{ED}), (\ref{f}) and (\ref{bound-f}) that
\begin{eqnarray}
\label{Fin}
&&E[\delta ^{L}(H_{L})^{2}] \le\sum_{i=0}^{n}E\left[ ( D_{T}^{2L-i}F) ^{2}\right] {\binom{%
L}{i}}^{2}i! \nonumber\\
&&~~\times
\int_{[t,T]^{2L-i}}f(x_{1},\ldots,x_{L-i},s_{1},\ldots,s_{L})(\,%
\mathrm{d}x)^{\otimes (L-i)}(\mathrm{d}s)^{\otimes L}\nonumber\\
&&\le\sum_{i=0}^{L}E\left[ ( D_{T}^{2L-i}F) ^{2}\right] {\binom{%
L}{i}}^{2}i!\left( (L-i)!\right) ^{2}i!\Big( \frac{L!}{i!(L-i)!}\Big) ^{2}
\nonumber\\
&&~~\times\int_{t\leq x_{1}\leq \ldots\leq x_{L-i}\leq T,\ t\leq s_{1}\leq \ldots\leq
s_{L-i}\leq T,\ t\leq s_{L-i+1}\leq \ldots\leq s_{L}\leq T}\ (\,\mathrm{d}%
x)^{\otimes (L-i)}(\mathrm{d}s)^{\otimes L} \nonumber\\
&&=\sum_{i=0}^{L}E\left[ ( D_{T}^{2L-i}F) ^{2}\right] {\binom{L}{i%
}}^{4}i!(T-t)^{2L-i}.
\end{eqnarray}%
Finally, combining (\ref{DefR}), (\ref{Fin}) and  Condition (\ref{Condi}) implies
\begin{equation*}
\label{RL}
E\left[( R_{[t,T]}^{L}) ^{2}\right]\leq \sum_{i=0}^{L}E\left[ (
D_{T}^{2L-i}F) ^{2}\right] {\binom{L}{i}}^{4}\frac{i!}{\left(
L!\right) ^{2}}(T-t)^{2L-i}\xrightarrow[L\rightarrow \infty]{}0.
\end{equation*}
Lemma \ref{Remainder} is proved.
\end{proof}
In view of (\ref{DBTE1}) and Lemma \ref{Remainder}, we claim that
 \begin{equation}
\label{DBTE2}
E[F|\mathcal{F}_{t}]=\sum_{l=0}^{\infty}(-1)^{l}\int_{t}^{T}\int_{s_{1}}^{T}\ldots%
\int_{s_{l-1}}^{T}E[D_{T}^{l}F|\mathcal{F}_{T}]\delta W(s_{l})\ldots\delta
W(s_{1})~\mbox{in}~L^2(\Omega),
\end{equation}
with convention that the first term in the above series is $E[F|\mathcal F_T]$.
In order to establish (\ref{aga2}), we rely on Lemma \ref{iSI} below.
\begin{lemma}
\label{iSI}
Let $F$ be defined as in Theorem \ref{BTE} and $m\in\{0,\ldots,M-1\}$, $t=m\Delta$, $T=(m+1)\Delta$. Then for any integer $l\ge0$,
\begin{equation}
\label{ite}
\int_{t\leq s_{1}\leq ...\leq s_{l}\leq T}F\,(\delta
W(s))^{\otimes l}=\sum_{i=0}^{l}D_{T}^{i}F\frac{(-1)^{i}(T-t)^{\frac{l+i}{2}%
}}{i!(l-i)!}h_{l-i}\Big(\frac{W(T)-W(t)}{\sqrt{T-t}}%
\Big),
\end{equation}
where we recall that $h_{l-i}$ denotes the Hermite polynomial of degree $l-i$ (see (\ref{hermite})).
\end{lemma}
\begin{proof}
For simplifying notation we define, for any integer $l\geq 1$,
$$
a_{l}(t):=\int_{t\leq s_{1}\leq ...\leq s_{l}\leq T}F\,(\,\mathrm{\delta }
W(s))^{\otimes l}.
$$
Our strategy of proving lemma \ref{iSI} is by induction. We first recall the Skorohod integral of a process multiplied by a random variable formula (see e.g. (1.49) in \cite{r8}): for a random variable $F\in \mathbb D_{\infty}([0,T])$ and a process $u$ such that $E[F^2\int_0^Tu(t)^2\ud t]$, we have for $0\le t\le T$,
\begin{equation}
\label{Skorohod}
\int_{t}^{T}Fu(s)\,\mathrm{\delta } W(s)=F\int_{t}^{T}u(s)\delta W(s)-\int_{t}^{T}D_{s}Fu(s)\mathrm{d}s.
\end{equation}
 When $k=1$, by (\ref{Skorohod}), the fact that $D_sF=D_TF$ for $s\in(t,T]$ and the definition of Hermite polynomials, we show (\ref{ite}) holds:
\begin{eqnarray*}
a_{1}(t) &=&\int_{t}^{T}F\,\mathrm{\delta}W(s) \\
&=&F\int_{t}^{T}\mathrm{d}W(s)-\int_{t}^{T}D_{s}F\mathrm{d}%
s \\
&=&\sum\limits_{i=0}^{1}D_{T}^{i}F\frac{( -1) ^{i}\left( T-t\right) ^{\frac{%
 1+i }{2}}}{i!( 1-i) !}h_{1-i}\Big( \frac{W(T)-W(t)}{%
\sqrt{ T-t}}\Big).
\end{eqnarray*}%
Now assume that (\ref{ite}) holds for $a_{l}(t)$ with some integer $l\ge1$. Recall that (see e.g. \cite{r10})
\begin{equation}
\label{Hermite1}
\frac{(T-t)^{\frac{l-i}{2}}}{(l-i)!}h_{l-i}\Big( \frac{W(T)-W(t)}{\sqrt{ T-t}}\Big)=\int_{t\leq s_{1}\leq s_{2}\leq
...\leq s_{l-i}\leq T}(\ud W(s))^{\otimes (l-i)}.
\end{equation} Therefore we write
\begin{eqnarray}
\label{inductionhyp}
a_{l}(t) &=&\sum\limits_{i=0}^{l}D_{T}^{i}F\frac{( -1) ^{i}\left(
T-t\right) ^{\frac{l+i }{2}}}{i!\left( l-i\right) !}%
h_{l-i}\Big( \frac{W(T)-W(t)}{\sqrt{T-t}}\Big)
\nonumber  \\
&=&\sum\limits_{i=0}^{l}\frac{(-1)^i(T-t)^i}{i!}D_{T}^{i}F\int_{t\leq s_{1}\leq s_{2}\leq
...\leq s_{l-i}\leq T}(\ud W(s))^{\otimes (l-i)}.
\end{eqnarray}%
(\ref{inductionhyp}) then yields
\begin{eqnarray}
\label{ak1}
a_{l+1}(t)&=&\int_{t}^{T}\Big( \int_{s_{1}\leq s_{2}\leq ...\leq s_{l+1}\leq
T}F\delta W(s_{l+1})\ldots\delta W(s_2))\Big) \,\mathrm{\delta }W(s_{1})\nonumber\\
&=&\int_{t}^{T}a_l(s_1) \,\mathrm{\delta }W(s_{1})\nonumber\\
&=&\sum\limits_{i=0}^{l}\frac{(-1)^i}{i!}\int_t^T\big(D_{T}^{i}F\big)(T-s_1)^i\Big(\int_{s_1\leq s_{2}\leq
...\leq s_{l+1-i}\leq T}\!\!\!\!\!\!\ud W(s_{l+1-i})\ldots\ud W(s_2)\Big)\delta W(s_1).\nonumber\\
\end{eqnarray}
For each $i=0,\ldots,l$, using again (\ref{Skorohod}), the fact that $D_sD_T^iF=D_T^{i+1}F$ for $s\in(t,T]$ and Fubini's theorem, we obtain
\begin{eqnarray}
\label{Skorohod1}
&&\int_t^T\big(D_{T}^{i}F\big)(T-s_1)^i\Big(\int_{s_1\leq s_{2}\leq
...\leq s_{l+1-i}\leq T}\ud W(s_{l+1-i})\ldots\ud W(s_2)\Big)\delta W(s_1)\nonumber\\
&&=\big(D_{T}^{i}F\big)\int_{t\leq s_{1}\leq
...\leq s_{l+1-i}\leq T}(T-s_1)^i(\ud W(s))^{\otimes (l+1-i)}\nonumber\\
&&~~-\big(D_{T}^{i+1}F\big)\int_{t\le s_1\le\ldots\le s_{l+1-i}\le T}(T-s_1)^i\ud W(s_{l+1-i})\ldots\ud W(s_2)\ud s_1\nonumber\\
&&=\big(D_{T}^{i}F\big)\int_{t\leq s_{1}\leq
...\leq s_{l+1-i}\leq T}(T-s_1)^i(\ud W(s))^{\otimes (l+1-i)}\nonumber\\
&&~~+\big(D_{T}^{i+1}F\big)\int_{t\le s_1\le\ldots\le s_{l-i}\le T}\frac{(T-s_1)^{i+1}}{i+1}(\ud W(s))^{\otimes (l-i)}\nonumber\\
&&~~-\big(D_{T}^{i+1}F\big)\int_{t\le s_1\le\ldots\le s_{l-i}\le T}\frac{(T-t)^{i+1}}{i+1}(\ud W(s))^{\otimes (l-i)}.
\end{eqnarray}
It follows from (\ref{ak1}), (\ref{Skorohod1}) and (\ref{Hermite1}) that (\ref{ite}) holds for $l+1$, hence it holds for all integer $l\ge1$.  Lemma \ref{iSI}
is proved.
\end{proof}
Then by applying Lemma \ref{iSI} to each term in (\ref{DBTE2}), we obtain:
\begin{eqnarray}
\label{jiandan}
&&E[F|\mathcal{F}_t]=\sum_{l=0}^{\infty}\sum_{i=0}^l\frac{(-1)^{i+l}(T-t)^{\frac{l+i}{2}%
}}{i!(l-i)!}h_{l-i}\Big(\frac{W(T)-W(t)}{\sqrt{T-t}}
\Big)E[D_{T}^{i+l}F|\mathcal{F}_T]
\end{eqnarray}
and (\ref{aga2}) follows by making the change of variable $l=k-i$ in (\ref{jiandan}). Theorem \ref{BTE} is proved. $\square$
\subsection{Proof of Theorem \ref{TEE}}
Given $F\in L^2(\Omega)$ is $\mathcal F_T$-measurable. Let $\Delta=T/M$ and suppose $F$ is generated on $(W(\Delta),\ldots,W(M\Delta))$. Fix $s\le T$.  Now  we are going to prove the following: for $t\le s$,
\begin{equation}
\frac{E[F|\mathcal{F}_{s-\delta }](\omega ^{t})-E[F|\mathcal{F}_{s}](\omega
^{t})}{\delta }\xrightarrow[\delta\rightarrow 0]{L^2(\Omega)}\frac{1}{2}%
\left( D_{s}^{2}E[F|\mathcal{F}_{s}]\right) (\omega ^{t}).  \label{result}
\end{equation}%
\begin{proof}
Denote by $\{\delta
_{k}=\Delta /k\}_{k\ge1}$. Let $m=\lfloor
s / \Delta \rfloor $, thus $s\in \lbrack m\Delta
,(m+1)\Delta ]$. First, suppose that $s\in (m\Delta ,(m+1)\Delta ]$. Then there exists $K>0$ such that for all $k\geq K$, $s-\delta_k \in (m\Delta, (m+1)\Delta)$.  Similar to (\ref{DBTE1}), we compute $E[F|\mathcal{F}_{s-\delta_k}]$ as
\begin{eqnarray}
\label{series}
&&E[F|\mathcal{F}_{s-\delta _{k}}]=E[F|\mathcal{F}_{s}]-\int_{s-\delta _{k}}^{s}E[D_{s}F|%
\mathcal{F}_{s}]\delta W(s_{1})\nonumber\\
&&~~+\int_{s-\delta
_{k}}^{s}\int_{s_{1}}^{s}E[D_{s}^{2}F|\mathcal{F}_{s}]\delta
W(s_{2})\delta W(s_{1})-R_{[s-\delta _{k},s]}^{3},
\end{eqnarray}
where
\begin{equation*}
R_{[s-\delta _{k},s]}^{3}=\int_{s-\delta
_{k}}^{s}\int_{s_{1}}^{s}\int_{s_{2}}^{s}E[D_{s}^{3}F|\mathcal{F}%
_{s_{3}}]\delta W(s_{3})\delta W(s_{2})\delta W(s_{1}).
\end{equation*}
On one hand, by Lemma \ref{Remainder},
\begin{equation}
E\left[( R_{[s-\delta _{k},s]}^{3}) ^{2}\right]\leq
\sum_{i=0}^{3}E\left[ ( D_{s}^{6-i}F) ^{2}\right] {\binom{3}{i}}%
^{4}\frac{i!}{\left( 3!\right) ^{2}}\delta _{k}^{6-i}.  \label{remainder}
\end{equation}
The above upper bound is finite, due to the fact that $F\in\mathbb D^6([0,T])$. On the other hand, when acted on freezing path operator, the terms in (\ref{series}) become
\begin{eqnarray}
&&\Big(-\int_{s-\delta _{k}}^{s}E[D_{s}F|%
\mathcal{F}_{s}]\delta W(s_{1})\Big) (\omega^t)=\delta_kE[D_s^2F|\mathcal{F}_s](\omega^t); \label{541}\\
&&\Big(\int_{s-\delta
_{k}}^{s}\int_{s_{1}}^{s}E[D_{s}^{2}F|\mathcal{F}_{s}]\delta
W(s_{2})\delta W(s_{1})\Big)(\omega^t)=\delta_k\Big(-\frac{1}{2}E[D_s^2F|\mathcal{F}_s]+\frac{\delta_k}{2}E[D_s^4F|\mathcal{F}_s]\Big)(\omega^t).\nonumber\\
\label{542}
\end{eqnarray}
Thus combining (\ref{series})-(\ref{542}) and the fact that $F\in \mathbb{D}^6([0,T])$, we get
\begin{equation*}
\frac{E[F|\mathcal{F}_{s-\delta _{k}}](\omega
^{t})-E[F|\mathcal{F}_{s}](\omega
^{t})}{\delta _{k}}-\frac{1}{2} D_{s}^{2}E[F|%
\mathcal{F}_{s}](\omega
^{t})]\xrightarrow[k\rightarrow \infty]{L^2(\Omega)}0.
\end{equation*}
Suppose now that $s=m\Delta $. Then similarly we can choose $K$ such that when $k>K$, $s-\delta_k\in ((m-1)\Delta,m\Delta)$ and then clearly (\ref{series})-(\ref%
{542}) also hold. Thus the proof is completed.
\end{proof}
\subsection{Proof of Theorem \ref{ExpFormula}}
For $F\in \mathbb D_\infty([0,T])$, there exists $G$ such that $F=G(W\chi_{[0,T]})$. We first construct a sequence $\{F^{(M)}=G_M(W\chi_{[0,T]})\}_{M\ge1}$ satisfying
 $$
 G_M(W\chi_{[0,t]})\xrightarrow[M\rightarrow\infty]{L^2(\Omega)}G(W\chi_{[0,t]}),~\mbox{for all}~t\in[0,T].
 $$
Since $\{W(t)\}_{t\ge0}$ is a continuous Gaussian process, it can be showed by the Stone-Weierstrass theorem and Wiener chaos decomposition that there exists a sequence of polynomial functions $\{p_M\}_{M \geq 1}$ such that
$$
p_M\Big(\int_0^th_1^{(M)}(s)\ud W(s),\ldots,\int_0^th_{n_M}^{(M)}(s)\ud W(s)\Big)\xrightarrow[M\rightarrow\infty]{L^2(\Omega)}G(W\chi_{[0,t]}),~\mbox{for all}~t\in[0,T],
$$
 where $n_M$ is some positive integer only depending on $M$ and $h_1^{(M)},\ldots,h_{n_M}^{(M)}\in L^2([0,T])$. Also observe that, each Wiener integral $\int_0^th_i^{(M)}(s)\ud W(s)$ has a "Riemann sum" approximation as
 $$
 J_M^{(i)}\Big(W\big(\frac{T}{M}\wedge t\big),W\big(\frac{2T}{M}\wedge t\big),\ldots,W(T\wedge t)\Big)\xrightarrow[M\rightarrow\infty]{L^2(\Omega)}\int_0^th_i^{(M)}(s)\ud W(s),~\mbox{for all}~t\in[0,T]
 $$
with $\{J_M^{(i)}\}_{M\geq 1}$ being polynomials. Therefore by the continuity of the Brownian paths and polynomials, we have for all $t\in[0,T]$,
 \begin{eqnarray*}
 G_M(W\chi_{[0,t]})&:=&\big(p_M\circ(J_M^{(1)},\ldots,J_M^{(n_M)})\big)\Big(<\chi_{[0,\frac{T}{M}]},W\chi_{[0,t]}>,\ldots,<\chi_{[0,T]},W\chi_{[0,t]}>\Big)\\
 &&\xrightarrow[M\rightarrow\infty]{L^2(\Omega)}G(W\chi_{[0,t]}).
 \end{eqnarray*}
 For $t\in[0,T]$, define $F^{(M)}=G_M(W\chi_{[0,T]})$. We remark that Theorem \ref{TEE} in fact holds for any cylindrical random variable $F=g(W(t_1),\ldots,W(t_n))$. Since $F^{(M)}$ is some polynomial of a discrete Brownian path, it belongs to $\mathbb D^6([0,T])$. Then from (%
\ref{result}), we obtain: for $s\ge t$,
\begin{equation}
P_{s}F^{(M)}(\omega ^{t})=F^{(M)}(\omega ^{t})+\frac{1}{2}\int_{s}^{T}(D_{u}^{2}\circ
P_{u})F^{(M)}(\omega ^{t})\,\mathrm{d}u.  \label{int}
\end{equation}%
For positive integer $n\ $we define the operator $%
T_{s}^{(n)}$ by
\begin{equation*}
T_{s}^{(n)}F^{(M)}:=\sum_{i=0}^{n}\mathcal{A}_{i,s}F^{(M)},
\end{equation*}%
where%
\begin{equation*}
\mathcal{A}_{i,s}F^{(M)}:=\frac{1}{2^{i}}\int_{s}^{T}\ldots \int_{s_{i-1}}^{T}%
D_{s_{1}}^{2}\ldots D_{s_{i}}^{2}F^{(M)}\,\mathrm{d}s_{i}\ldots \,\mathrm{d}s_{1}.
\end{equation*}%
Then by iterating (\ref{int}) we obtain: for $n\ge1$,
\begin{equation}
\label{it}
P_{s}F^{(M)}(\omega ^{t})=T_{s}^{(n-1)}F^{(M)}(\omega ^{t})+\frac{1}{2^{n}}%
\int_{s}^{T}\ldots \int_{u_{n-1}}^{T}(D_{u_{n}}^{2}\ldots D_{u_{1}}^{2}\circ
P_{u_{n}})F^{(M)}(\omega ^{t})\,\mathrm{d}u_{n}\ldots \,\mathrm{d}u_{1}.
\end{equation}
Remark that for a general $V$ of the form $V(W\chi_{[0,T]})=f(W(t_1),\ldots,W(t_n))$, we have
\begin{eqnarray*}
D_u(V(W\chi_{[0,T]}))&=&(\mathcal D_u\circ V)(W\chi_{[0,T]})\\
&:=&\sum_{i=1}^n\frac{\partial f}{\partial x_i}\Big(\int_0^{T}\chi_{[0,t_1]}(s)\ud W(s),\ldots,\int_0^{T}\chi_{[0,t_n]}(s)\ud W(s)\Big)\chi_{[0,t_i]}(u)
\end{eqnarray*}
is continuous with respect to $T$. We note here $\mathcal D_u$ is explicitly defined by: if
$$
V=f(<\chi_{[0,t_1]},\cdot>,\ldots,<\chi_{[0,t_n]},\cdot>),
$$
then
 $$
 \mathcal D_u\circ V=\sum_{i=1}^n\frac{\partial f}{\partial x_i}(<\chi_{[0,t_1]}, \cdot>,\ldots,<\chi_{[0,t_n]},\cdot>)\chi_{[0,t_i]}(u).
 $$
 Since $D_{u_{n}}^{2}\ldots
D_{u_{1}}^{2}G(W\chi_{[0,t]})\in L^2(\Omega)$ for all $u_1,\ldots,u_n\ge0$, then the closability of the Malliavin derivative operator (see Lemma 1.2.2 in \cite{r8}) implies the fact that for all $t\in[0,T]$,
 $$
 D_{u_{n}}^{2}\ldots
D_{u_{1}}^{2}(G_M(W\chi_{[0,t]}))\xrightarrow[M\rightarrow\infty]{L^2(\Omega)}D_{u_{n}}^{2}\ldots
D_{u_{1}}^{2}G(W\chi_{[0,t]}).
$$
Hence by remark (\ref{conv:omega}), we obtain
$$
(\mathcal D_{u_{n}}^{2}\circ\ldots
\circ\mathcal D_{u_{1}}^{2}\circ G_M(W\chi_{[0,T]}))(\omega^t)\xrightarrow[M\rightarrow\infty]{L^2(\Omega)}(\mathcal D_{u_{n}}^{2}\circ\ldots
\circ\mathcal D_{u_{1}}^{2}\circ G(W\chi_{[0,T]}))(\omega^t).
$$
Equivalently,
\begin{equation}
\label{preserve}
 (D_{u_{n}}^{2}\ldots
D_{u_{1}}^{2}F^{(M)})(\omega^t)\xrightarrow[M\rightarrow\infty]{L^2(\Omega)}(D_{u_{n}}^{2}\ldots
D_{u_{1}}^{2}F)(\omega^t).
\end{equation}
Thus according to assumption (\ref{assumptionb}) and (\ref{preserve}), we get
\begin{eqnarray}
\label{bounddiff}
&&\| (P_{s}-T_{s}^{(n-1)})F^{(M)}(\omega ^{t})\|_{L^2(\Omega)}\nonumber\\
&& =\Big\|\frac{1}{2^{n}}\int_{s}^{T}\ldots
\int_{u_{n-1}}^{T}(D_{u_{n}}^{2}\ldots D_{u_{1}}^{2}\circ P_{u_{n}})F^{(M)}(\omega
^{t})\,\mathrm{d}u_{n}\ldots \,\mathrm{d}u_{1}\Big\|_{L^2(\Omega)}\notag \\
&&\le\frac{(T-s)^{n}}{2^{n}n!}\Big(\Big\| \sup_{u_{1},\ldots
,u_{n}\in [ 0,T]}\left\vert (D_{u_{n}}^{2}\ldots
D_{u_{1}}^{2}F)(\omega^t)\right\vert\Big\|_{L^2(\Omega)}\nonumber\\
&&~~ +\Big\|\sup_{u_{1},\ldots
,u_{n}\in [ 0,T]}\left\vert (D_{u_{n}}^{2}\ldots
D_{u_{1}}^{2}F^{(M)})(\omega^t)-(D_{u_{n}}^{2}\ldots
D_{u_{1}}^{2}F)(\omega^t)\right\vert\Big\|_{L^2(\Omega)}\Big)\nonumber\\
&&~~ \xrightarrow[n\rightarrow
\infty]{}0.
\end{eqnarray}
We now take $s=t$ and let $n\rightarrow\infty$ in (\ref{it}). By (\ref{bounddiff}),
\begin{equation}
\label{FM}
E[F^{(M)}|\mathcal{F}_{t}]=P_{t}F^{(M)}=\sum_{n=0}^{\infty }\frac{1}{2^{n}}\int_{t}^{T}\ldots
\int_{s_{n-1}}^{T}(D_{s_{n}}^{2}\ldots D_{s_{1}}^{2}F^{(M)})(\omega ^{t})\,\mathrm{d}%
s_{n}\ldots \,\mathrm{d}s_{1}.
\end{equation}
Letting $M\rightarrow\infty$ in (\ref{FM}), then using (\ref{preserve}), dominated convergence theorem and assumption (\ref{assumptionb}), we obtain
$$
E[F|\mathcal{F}_{t}]=\sum_{n=0}^{\infty }\frac{1}{2^{n}}\int_{t}^{T}\ldots
\int_{s_{n-1}}^{T}(D_{s_{n}}^{2}\ldots D_{s_{1}}^{2}F)(\omega ^{t})\,\mathrm{d}%
s_{n}\ldots \,\mathrm{d}s_{1}.~~\square
$$
\section*{Acknowledgements}

We thank \emph{Josep Vives} and \emph{Hank Krieger} for reviewing our manuscript. All errors are ours.

\end{document}